\input ./style/arxiv-vmsta.cfg
\documentclass[numbers,compress,v1.0.1]{vmsta}
\usepackage{vtexbibtags}

\volume{6}
\issue{2}
\pubyear{2019}
\firstpage{167}
\lastpage{193}
\aid{VMSTA134}
\doi{10.15559/19-VMSTA134}
\articletype{research-article}

\startlocaldefs
\ifdefined\HCode 
\def\index#1{}
\else 
\fi
\theoremstyle{plain}
\newtheorem{thm}{Theorem}[section]
\newtheorem{proposition}{Proposition}[section]
\newtheorem{lem}[thm]{Lemma}

\theoremstyle{definition}
\newtheorem{defi}[thm]{Definition}
\newtheorem{rk}[thm]{Remark}

\newcommand{\rrVert}{\Vert}
\newcommand{\llVert}{\Vert}
\newcommand{\rrvert}{\vert}
\newcommand{\llvert}{\vert}
\allowdisplaybreaks
\endlocaldefs

\begin{document}

\begin{frontmatter}
\pretitle{Research Article}

\title{Moderate deviations for a stochastic Burgers equation}

\author[a]{\inits{R.}\fnms{Rachid}~\snm{Belfadli}\thanksref{cor1}\ead[label=e1]{r.belfadli@uca.ma}}
\author[b,c]{\inits{L.}\fnms{Lahcen}~\snm{Boulanba}\ead[label=e2]{l.boulanba@crmefsm.ac.ma}}
\author[d]{\inits{M.}\fnms{Mohamed}~\snm{Mellouk}\ead[label=e3]{mohamed.mellouk@parisdescartes.fr}}
\thankstext[type=corresp,id=cor1]{Corresponding author.}
\address[a]{D\'{e}partement de Math\'{e}matiques,\break
Laboratoire de M\'{e}thodes Stochastiques Appliqu\'{e}es\break \`{a} la Finance et l'Actuariat (LaMSAFA),\break
Facult\'{e} des Sciences et~Techniques Gu\'{e}liz,\break
\institution{Universit\'{e} Cadi Ayyad},\break
B.P. 549, 40 000 Marrakech, \cny{Maroc}}
\address[b]{D\'{e}partement de Math\'{e}matiques,\break
\institution{Centre R\'{e}gional des M\'{e}tiers\break  de l'\'{E}ducation et~de~la~Formation-Souss Massa},\break
B.P. 106, 86 350 Inezgane, \cny{Maroc}}
\address[c]{Laboratoire LISTI, ENSA,
\institution{Universit\'{e} Ibn Zohr},\break
B.P. 1136, 80000 Agadir, \cny{Maroc}}
\address[d]{MAP5, CNRS UMR 8145,
\institution{Universit\'{e} Paris Descartes},\break
45, rue des Saints-P\`{e}res,\break
75270 Paris Cedex 6, \cny{France}}



\markboth{R. Belfadli et al.}{Moderate deviations for a stochastic Burgers equation}

\begin{abstract}
A moderate deviations\index{moderate deviations}  principle for the law of a
stochastic Burgers equation is proved via the weak
convergence approach.
In addition, some useful estimates
toward a central limit theorem
are 
established.
\end{abstract}

\received{\sday{29} \smonth{10} \syear{2018}}
\revised{\sday{29} \smonth{3} \syear{2019}}
\accepted{\sday{17} \smonth{4} \syear{2019}}
\publishedonline{\sday{16} \smonth{5} \syear{2019}}
\end{frontmatter}

\section{Introduction}\label{sec1}

We consider the following stochastic Burgers equation\index{stochastic Burgers equation} with
multiplicative space-time white noise, indexed by $\varepsilon >0$,
given by
%
\begin{align}
\label{1} \dfrac{\partial u^{\varepsilon }}{\partial t}(t,x)&= \Delta u^{\varepsilon
}(t,x)+
\dfrac{1}{2}\dfrac{\partial }{\partial x} \bigl(u^{\varepsilon
}(t,x)
\bigr)^{2}
\nonumber
\\
&\quad +\sqrt{\varepsilon }\sigma \bigl(u^{\varepsilon }(t,x)\bigr) {
\dot{W}}(t,x), \quad (t,x)\in [0,T]\times [0, 1],
\end{align}
with Dirichlet\index{Dirichlet boundary conditions}'s boundary conditions $u^{\varepsilon }(t,0)=u^{\varepsilon
}(t,1)=0$ for $t\in [0,T]$, and the initial condition $u^{\varepsilon
}(0,x)=u_{0}(x)$ for $x\in [0,1]$. We assume that $u_{0}$ is continuous
on $[0 , 1]$ and $\sigma $ is bounded and globally Lipschitz\index{globally Lipschitz} on
$\mathbb{R}$. The driving noise ${W}$ is a space-time Brownian sheet\index{Brownian sheet}
defined on some filtered probability space
$(\varOmega, {\mathcal{F}},\allowbreak ({\mathcal{F}}_{t})_{t\in [0,T]}, \mathbb{P})$.\looseness=1

The deterministic Burgers equation was introduced in
\cite{burgers1974nonlinear} as a simplified mathematical model
describing the turbulence phenomena in fluids. Its stochastic version
has been the subject of several works; see for instance
\cite{bertini1994stochastic,gyongy1998existence,morien1999density}, and the references therein. In particular, a
large deviation principle is established in
\cite{setayeshgar2014large} for an ``{\textit{additive version}}'' of
(\ref{1}), and in \cite{cardon1999large} and
\cite{foondun2017large} for a class of Burgers' type stochastic partial
differential equations (SPDEs\index{SPDEs} for short) including (\ref{1}). Generally
speaking, large deviations theory deals with determining how fast the
probabilities\index{probabilities} $\mathbb{P}(A_{\varepsilon })$ of a family of rare events
$(A_{\varepsilon })$ decay to 0 as $\varepsilon $ tends to 0, and how
to compute the precise rate of decay as a function of the rare events.
A related natural important
question is to study moderate deviations \index{moderate deviations}
results which deals with probabilities\index{probabilities} of deviations of
``{\textit{smaller order}}'' than in large deviations. We \xch{will}{wil} precise
below the main difference between moderate and large deviations
principles in the context of stochastic Burgers equation,\index{stochastic Burgers equation} and for a
deeper description and detail about these two kinds of deviations
principles and their relationship, we refer the reader to
\cite{budhiraja2016moderate}.

Our first goal in this paper is to study the moderate deviations \index{moderate deviations} of
${u}^{\varepsilon }$ from the deterministic solution ${u}^{0}$ of the
equation (\ref{2}) below. More precisely, we deal with the deviations
of the trajectory
%
\begin{equation}
\label{16} \bar{u}^{\varepsilon }(t,x):= \frac{{u}^{\varepsilon }(t,x) -
{u}^{0}(t,x)}{a(
\varepsilon )},
\end{equation}
where the deviation scale $a: \mathbb{R}^{+} \longrightarrow
\mathbb{R}^{+}$\index{deviation scale} is such that
%
\begin{equation}
\label{15} a(\varepsilon ) \longrightarrow 0 \quad \text{and} \quad h(
\varepsilon ): = \frac{a(\varepsilon )}{\sqrt{\varepsilon }} \longrightarrow \infty , \quad \text{as} \
\varepsilon \longrightarrow 0,
\end{equation}
and $u^{0}$ stands for the solution of the following deterministic
partial differential equation
%
\begin{equation}
\label{2} \dfrac{\partial u^{0}}{\partial t}(t,x)= \dfrac{\partial ^{2} u^{0}}{
\partial x^{2}}(t,x)+
\dfrac{1}{2}\dfrac{\partial }{\partial x} \bigl(u ^{0}(t,x)
\bigr)^{2}, \quad (t,x)\in [0,T]\times [0, 1],
\end{equation}
with Dirichlet\index{Dirichlet boundary conditions}'s boundary conditions $u^{0}(t,0)=u^{0}(t,1)=0$ for
$t\in [0,T]$, and the initial condition $u^{0}(0,x)=u_{0}(x)$.

The deviation scale $a(\varepsilon )$\index{deviation scale} influences strongly the asymptotic
behavior of $\bar{u}^{\varepsilon }$. In fact, for certain norm
$\|\cdot \|$, bounds of the probabilities $\mathbb{P}  (\frac{\|u
^{\varepsilon } - u^{0}\|}{\sqrt{\varepsilon }} \in \cdot  )$\index{probabilities}
are dealt with the central limit\index{central limit theorem} theorem, while probabilities
$\mathbb{P}  (\|u^{\varepsilon } - u^{0}\| \in \cdot   )$\index{probabilities} are
estimated by large deviations results. Furthermore, when we are
interested in probabilities\index{probabilities} of the form $\mathbb{P}  (\frac{\|u
^{\varepsilon } - u^{0}\|}{a(\varepsilon )} \in \cdot   )$ under
the condition (\ref{15}) (e.g. $a(\varepsilon ) = \varepsilon ^{1/4}$),
then we are in the framework of the so-called moderate deviations \index{moderate deviations}  which
fills 
the gap between the central limit theorem\index{central limit theorem} scale ($a(\varepsilon
) = \sqrt{\varepsilon }$) and the large deviations scale ($a(\varepsilon
) = 1$). In this paper, we will establish the moderate deviations \index{moderate deviations}
principle for (\ref{1}). For the study of this topic for various kind
of stochastic processes, see for instance e.g.
\cite{chen2017well,de1992moderate,gao1996moderate,liming1995moderate}.

Furthermore, there are basically two approaches to analyzing moderate
and large deviations for processes. The former, which is originally used
by Freidlin and Wentzell \cite{freidlinrandom} for diffusions
processes, relies on discretization and localization arguments that
allow to deduce the large deviations principle, for the solutions of
equations under study, using a general contraction principle from some
Schilder type theorems for the driving noises. The second one, which we
are going to use in present paper, is the so-called  weak convergence\index{weak convergence}
approach. It was introduced in \cite{dupuis2011weak} and developed
in \cite{boue1998variational,budhiraja2000variational} and \cite{budhiraja2008large}, and
its starting point is the equivalence between large deviations principle
and Laplace principle\index{Laplace principle} in the setting of Polish spaces. It consists in
using certain variational formulas that can be viewed as the minimal
cost functions\index{minimal cost functions} for associated stochastic optimal problems. These minimal
cost functions\index{minimal cost functions} have a form to which the theory of weak convergence\index{weak convergence} of
probability measures can be applied. We refer to
\cite{dupuis2011weak} for a more complete exposition on this approach.

We stress here that, in the present paper, we mainly use
the weak convergence\index{weak convergence} approach to establish \textit{moderate deviations}\index{moderate deviations}
for stochastic Burgers' equations\index{stochastic Burgers equation} while in the previous works
(\cite{cardon1999large,setayeshgar2014large,foondun2017large}) the authors studied the large deviations
principle for this equation. The most likely advantage in using the weak
convergence\index{weak convergence} approach is that it allows one to avoid establishing
technical exponential-type probability estimates usually needed in the
classical studies of large deviations principle, and reduces the proofs
to demonstrating qualitative properties like existence, uniqueness and
tightness\index{tightness} of certain analogues of the original processes.

We also note that the 
greatest difficulty in studying any aspect of
Burgers' type equations lies in their quadratic term. In fact, most of
the techniques usually used to deal with stochastic differential
equations with Lipschitz drift coefficients don't longer work generally,
and one resort to localization or tightness argument to circumvent this
difficulty.

As pointed out before, we will prove a moderate deviations \index{moderate deviations}  principle for
the stochastic Burgers equation\index{stochastic Burgers equation} (\ref{1}), and two first-step results
toward a central limit\index{central limit theorem} theorem. It is worth bearing in mind that the
most difficulty we have encountered in establishing a central limit
theorem\index{central limit theorem} is mainly due to the quadratic term appearing in the Burgers
equation for which the classical conditions (namely, the Lipschitz
condition on the drift coefficient, the boundedness\index{boundedness} and the
differentiability of its derivative) are no longer satisfied.

The paper is organized as follows. Section \ref{sec2} is devoted to some
preliminaries. The framework of our moderate deviations \index{moderate deviations}  result and its
proof are given in Section \ref{sec3}. In Section \ref{sec4}, 
toward a central
limit theorem\index{central limit theorem} for the stochastic Burgers equation,\index{stochastic Burgers equation} we prove the uniform
boundedness and the convergence\index{convergence} of $u^{\varepsilon }$ to $u^{0}$ in the space
$\mathrm{L}^{q}(\varOmega ; C([0,T];{\allowbreak} \mathrm{L}^{2}([0,1])))$ for
$q\geqslant 2$. Furthermore, some technical results needed in our proofs
are included in the Appendix.

In this paper all positive constants are denoted by $c$, and their
values may change from line to line. Also, for $\rho \geqslant 1$ and
$t\in [0,T]$, the usual norms on $L^{\rho }([0,1])$ and $\mathcal{H}
_{t}: = L^{2}([0,t]\times [0,1])$ are respectively denoted by
$\|\cdot \|_{\rho }$ and $\|\cdot \|_{\mathcal{H}_{t}}$.

\section{Preliminaries}\label{sec2}

Let $  \{W(t,x), t\in [0,T],  x\in [0,1]   \}$ be a
space-time Brownian sheet\index{Brownian sheet} on a filtered probability space $(\varOmega ,
{\mathcal{F}}, {\mathcal{F}}_{t},\mathbb{P})$, that is, a zero-mean
Gaussian field with covariance function given by
\begin{equation*}
\mathbf{E} \bigl(W(t,x)W(s,y) \bigr) = (t\wedge s) (x\wedge y), \quad s, t \in
[0,T],\ x, y \in [0,1].
\end{equation*}

For each $t\in [0,T]$, ${\mathcal{F}}_{t}$ is the completion of the
$\sigma $-field generated by the family of random variables %
$\{W(s,x), 0\leqslant s \leqslant t, x\in [0,1]\}$.

A rigorous meaning to the solution of (\ref{1}) is given by a jointly
measurable and ${\mathcal{F}}_{t}$-adapted process $u^{\varepsilon }:
=\{u^{\varepsilon }(t,x);(t,x)\in [0, T]\times [0, 1]\}$ satisfying, for
almost all $\omega \in \varOmega $ and all $t\in [0,T]$ the following
evolution equation:
%
\begin{eqnarray}
\label{21} u^{\varepsilon }(t,x) \!\!\!\!&=&\!\!\!\!\int
_{0}^{1}G_{t}(x,y)u_{0}(y)dy
\,-\, \int_{0}^{t}\int_{0}^{1}
\partial _{y} G_{t-s}(x,y) \bigl(u^{\varepsilon }(s,y)
\bigr)^{2}dyds
\nonumber
\\
&&\!\!\!\!{} +\sqrt{\varepsilon }\int_{0}^{t}
\int_{0}^{1}G_{t-s}(x,y) \sigma
\bigl(u^{\varepsilon }(s,y)\bigr)W(ds,dy), \label{3}
\end{eqnarray}
for $dx$-almost all $x\in [0,T]$, where $G_{t}(\cdot ,\cdot )$ denotes
the Green kernel corresponding to the operator $\frac{\partial }{
\partial t}-\Delta $ with the Dirichlet boundary conditions.\index{Dirichlet boundary conditions} The
stochastic integral in (\ref{21}) is understood in the Walsh sense, see
\cite{walsh1986introduction}.

By Theorem 2.1 in \cite{gyongy1998existence}, there exists a unique
$L^{2}[0,1]$-valued continuous stochastic process $\{u^{\varepsilon
}(t,.), t\in [0,T]\}$ satisfying the equation (\ref{3}).

The deterministic equation (\ref{2}) obtained when the parameter
$\varepsilon $ tends to zero can be written in the following integral
form
%
\begin{align}
\label{4} u^{0}(t,x) =\int_{0}^{1}G_{t}(x,y)u_{0}(y)dy
\,-\,\int_{0}^{t}\int_{0}^{1}
\partial _{y} G_{t-s}(x,y) \bigl(u^{0}(s,y)
\bigr)^{2}dyds.
\end{align}
Since (\ref{4}) corresponds to $\sigma \equiv 0$ in the degenerate case
studied in \cite{gyongy1998existence}, it admits a unique solution
$u^{0}$ belonging to $C([0,T]; \mathrm{L}^{2}([0,1]))$. Moreover, the
continuity\index{continuity} of $u^{0}$ on the compact set $[0,T]$ implies that
%
\begin{equation}
\label{6bis} \sup_{t\in [0, T]} \bigl\llVert u^{0}(t,
\cdot ) \bigr\rrVert _{2}^{q} < \infty ,
\end{equation}
for all $q\geqslant 2$.

We now recall some estimations of the Green kernel function $G$, as
stated in \cite{gyongy1998existence} and
\cite{morien1999density}, that will be used in the sequel.

\begin{lem}
\label{20}Let $G$ denotes the Green kernel corresponding to the operator $\frac{\partial }{\partial t}-\Delta$ with the Dirichlet boundary conditions. Then, we have
\begin{itemize}%
\item[i)] for any $t\in [0, T]$ and $y\in [0, 1]$:\quad $\displaystyle\int_{0}^{1}G_{t}(x,y)dx = 1 $;
\item[ii)] for any $t\in [0, T]$ and $\dfrac{1}{2} < \beta < \dfrac{3}{2}$:\quad  $\displaystyle
\int_{0}^{t}\int_{0}^{1}  \llvert \partial _{x} G_{t-s}(x,y) \rrvert ^{
\beta }dxds \leqslant c_{\beta , T}$,\vspace*{3pt}\quad where $c_{\beta , T}$ is a
constant depending only on $T$ and $\beta $.
\end{itemize}
Moreover, there exists a constant $c$, depending only on $T$, such that for all
$y, z \in [0, 1]$ and $ t , t' \in [0,T]$ such that $0\leqslant t
\leqslant t' \leqslant 1$,
\begin{itemize}
\item[iii)]$
\displaystyle
\int_{t}^{t'}\int_{0}^{1}G_{t'-s}^{2}(x,y)dxds \leqslant c
\sqrt{t'-t}$\quad  and\quad  $
\displaystyle
\int_{0}^{t}\int_{0}^{1}G_{t-s}^{2}(x,y)dxds \leqslant c$;
\item[iv)]$
\displaystyle
\int_{0}^{ t'}\int_{0}^{1}[G_{t-s}(x,y) - G_{t'-s}(x,y)]^{2}dxds
\leqslant c \sqrt{t'-t}$;
\item[v)] $
\displaystyle
\int_{0}^{t}\int_{0}^{1}[G_{s}(x,y) - G_{s}(x,z)]^{2}dxds \leqslant c
|y-z|$.
\end{itemize}
\end{lem}

\section{Moderate deviations \index{moderate deviations} }\label{sec3}

\subsection{Framework and the main result}

According to Varadhan \cite{varadhan1966asymptotic} and
\cite{bryc1992large}, a crucial step toward the large deviations principle
is the Laplace principle.\index{Laplace principle} Therefore, we will focus later on establishing
this principle which we formulate in the following

\begin{defi}[Laplace principle] A family of random variables $\{X^{\varepsilon };
\, \varepsilon >0\}$ defined on a Polish space $\mathcal{E}$, is said
to satisfy the Laplace principle\index{Laplace principle} with speed $\lambda ^{2}(\varepsilon
)$ and rate function $I:{\mathcal{E}}\longrightarrow [0, \infty ]$ if
for any bounded continuous function $F: \mathcal{E} \rightarrow
\mathbb{R}$, we have
\begin{equation*}
\lim_{\varepsilon \rightarrow 0}\lambda ^{2}(\varepsilon )\log
\mathbf{E} \biggl( \exp \biggl[-\frac{1}{\lambda ^{2}(\varepsilon )}F\bigl(X ^{\varepsilon }
\bigr) \biggr] \biggr)= -\inf_{f\in {\mathcal{E}}}\bigl\{F(f)+I(f) \bigr\},
\end{equation*}
where $\mathbf{E}$ is the expectation with respect to $P$.
\end{defi}
In the context of the weak convergence\index{weak convergence} approach, 
the proof of the Laplace
principle\index{Laplace principle} for functionals of the Brownian sheet\index{Brownian sheet} is essentially based on
the following variational representation formula, which was originally
proved in \cite{budhiraja2000variational}.

\begin{thm}
Let $f\in \mathcal{C}([0,T]\times [0,1]; \mathbb{R})\longrightarrow
\mathbb{R}$ be a bounded measurable mapping $\mathcal{C}([0,T]\times
[0,1]; \mathbb{R})$ into $\mathbb{R}$, and let $\mathcal{P}_{2}$ be the
class of all predictable processes $u$ such that $\|u\|_{\mathcal{H}
_{T}} < \infty , a.s$. Then
%
\begin{equation}
-\log \mathbf{E}\exp \bigl\{-f(B)\bigr\} = \inf_{u\in \mathcal{P}_{2}}
\biggl(\frac{1}{2} \llVert u \rrVert _{\mathcal{H}_{T}}^{2} +
f \bigl(B^{u} \bigr) \biggr),
\end{equation}
where
\begin{equation*}
B^{u}(t,x):= B(t,x) + \int_{0}^{t}\int_{0}^{x}u(s,y)dyds,\quad \mbox{for any}\ (t,x)\in [0,T]\times [0,1].
\end{equation*}
\end{thm}

\subsubsection{Sufficient conditions for the general Laplace principle\index{Laplace principle}}

Here, we briefly describe the result needed, in our context, for proving
the Laplace principle,\index{Laplace principle} and state our main result.

Let us first introduce some notations. For $\varepsilon >0$, denote by
$\mathcal{G}^{\varepsilon }: \mathcal{E}_{0} \times \mathcal{C} ([0,T]
\times [0,1]; \mathbb{R}) \rightarrow \mathcal{E}$ a measurable map,
where $\mathcal{E}_{0}$ stands for a compact subspace of $\mathcal{E}$
in which the initial condition $u_{0}$ takes values, and let
%
\begin{equation}
\label{BB} X^{\varepsilon , u_{0}}: = \mathcal{G}^{\varepsilon }
\bigl(u_{0}, h(\varepsilon )W\bigr).
\end{equation}
Later, we will state sufficient conditions for the Laplace principle\index{Laplace principle} for
$X^{\varepsilon , u_{0}}$ to hold uniformly in $u_{0}$ for compact
subsets of $\mathcal{E}_{0}$.

For any positive integer $N$, we introduce
\begin{equation*}
S^{N} := \bigl\{\phi \in \mathcal{H}_{T}
: \llVert \phi \rrVert _{\mathcal{H}
_{T}}^{2} \leqslant N \bigr\}
\end{equation*}
and
\begin{equation*}
\mathcal{P}_{2}^{N} := \bigl\{v(\omega ) \in
\mathcal{P}_{2}: v( \omega ) \in S^{N}, P\mbox{-a.s.} \bigr\}.
\end{equation*}
It is worth noticing that the space
$S^{N}$ is a compact metric space equipped with the weak topology\index{weak topology} on
$L^{2}([0,T]\times [0,1])$ and that $\mathcal{P}_{2}^{N}$ is the space
of controls, which plays a central role in the weak convergence\index{weak convergence} approach.

For $u\in \mathcal{H}_{T}$, define the element $\mathcal{I}(u)$ in
$\mathcal{C} ([0,T]\times [0,1]; \mathbb{R}) $ by
\begin{equation*}
\mathcal{I}(u) (t,x):=\int_{0}^{t} \int
_{0}^{x} u(s,y) ds dy, \quad t\in [0,T], \ x \in [0,1].
\end{equation*}

We are now in position to introduce the following result, due to
Budhiraja \xch{et al.}{and al.} \cite{budhiraja2008large}, ensuring sufficient
condition for the Laplacian principle to hold.
%
\begin{proposition}[Theorem $7$ in \cite{budhiraja2008large}]
\label{48}
Assume that there exists a measurable 
\begin{equation*}
\mathcal{G}^{0}: \mathcal{E}_{0} \times \mathcal{C}
\bigl([0,T]\times [0,1]; \mathbb{R}\bigr) \rightarrow \mathcal{E},
\end{equation*}
such that the following hold:
\begin{itemize}%
\item[\textbf{(A1)}] For any integer $M>0$, any family $\{v^{\varepsilon
}; \varepsilon > 0\} \subset {\mathcal{P}_{2}}^{M}$ and $\{ u_{0}^{
\varepsilon }\} \subset \mathcal{E}_{0}$ such that $v^{\varepsilon }
\rightarrow v $ and $u_{0}^{\varepsilon }\rightarrow u_{0}$ in
distribution (as $S^{N}$-valued random elements), as $\varepsilon
\rightarrow 0$. Then ${\mathcal{G}}^{\varepsilon }(u_{0}^{\varepsilon
}, W + h(\varepsilon )\mathcal{I}(v^{\varepsilon })) \rightarrow
\mathcal{G}^{0}(u_{0},\mathcal{I}(u))$, in distribution as $\varepsilon
\rightarrow 0$;
\item[\textbf{(A2)}] For any integer $M\,{>}\,0$ and compact set
$K\,{\subset}\,\mathcal{E}_{0}$, the set
$\varGamma _{M,K}\,{:=}\,\{ \mathcal{G}^{0}(u_{0},\mathcal{I}(u));\allowbreak
u \in S^{M}, u_{0} \in K  \}$ is a compact subset of $\mathcal{E}$.
\end{itemize}
Then, the family $\{X^{\varepsilon , u_{0}};\, \varepsilon >0\}$ defined
by (\ref{BB}) satisfies the Laplace principle\index{Laplace principle} on $\mathcal{E}$ with
speed $\lambda ^{2}(\varepsilon )$ and rate function $I_{u_{0}}$ given,
for any $h \in \mathcal{E}$ and $u_{0} \in \mathcal{E}_{0}$, by
%
\begin{equation}
\label{rate} I_{u_{0}}(h):= \inf_{\{ v\in \mathcal{H}_{T}: h=
\mathcal{G}^{0}(u_{0}, \mathcal{I}(v))
\}} \biggl\{
\frac{1}{2} \llVert v \rrVert _{\mathcal{H}_{T}}^{2} \biggr
\},
\end{equation}
where the infimum over an empty set is taken to be $\infty $.
\end{proposition}

\subsubsection{Controlled processes for SPDEs\index{SPDEs} (\ref{1})}

In this subsection, we adapt the general scheme described above to study
moderate deviations \index{moderate deviations}  for the equation (\ref{1}).

We 
denote by $\mathcal{E}= \mathcal{E}_{0}:= C([0,T]; L^{2}([0,1]))$ 
the space of solutions of (\ref{1}). As we are interested in proving
the Laplace principle\index{Laplace principle} for $\bar{u}^{\varepsilon }(t,x)$ defined by
(\ref{16}), we interpret $\bar{u}^{\varepsilon }$ as a functional of the
Brownian sheet\index{Brownian sheet} $W$. Indeed, using (\ref{3}) and (\ref{4}) we deduce that
$\bar{u}^{\varepsilon }(t,x)$ satisfies for all $\omega \in \varOmega $ and
all $t\in [0,T]$ the following equation
%
\begin{align}\label{51}
\bar{u}^{\varepsilon }(t,x)
& = \frac{1}{h(\varepsilon )}\int_{0} ^{t}\int_{0}^{1}G_{t-s}(x,y)\sigma
\bigl(u^{0}(s,y) + \sqrt{\varepsilon }h(\varepsilon) \bar{u}^{\varepsilon }(s,y)\bigr)W(dy,ds)\nonumber\\
&\quad{}- {\sqrt{\varepsilon }h(\varepsilon )}\int_{0}^{t}\int_{0}^{1}
\partial _{y} G_{t-s}(x,y) \bigl[\bar{u}^{\varepsilon }(s,y)\bigr] ^{2}dyds\nonumber\\
&\quad{}- 2\int_{0}^{t}\int_{0}^{1} \partial _{y} G_{t-s}(x, y) \bar{u}^{\varepsilon }(s,y)u^{0}(s, y)dyds,
\end{align}
for $dx$-almost all $x\in [0,T]$.

This implies (see Theorem IV.9.1. of \cite{ikeda2014stochastic})
the existence of a measurable mapping
\begin{equation*}
\mathcal{G}^{\varepsilon }: C\bigl([0, 1]; \mathbb{R}\bigr)\times C\bigl([0, T]
\times [0, 1];\mathbb{R}\bigr)\rightarrow C\bigl([0, T];L^{2}
\bigl([0, 1]\bigr)\bigr),
\end{equation*}
such that
\begin{equation*}
\bar{u}^{\varepsilon } = \mathcal{G}^{\varepsilon }(u_{0},W).
\end{equation*}
As a first step toward the conditions \textbf{(A1)} and \textbf{(A2)}
stated in Proposition \ref{48}, we define for $ v^{\varepsilon }
\in {\mathcal{P}}_{2}^{N}$,
%
\begin{equation}
\label{69} \bar{u}^{\varepsilon ,v^{\varepsilon }} := \mathcal{G}^{\varepsilon }\bigl(u
_{0}, W + h(\varepsilon )\mathcal{I}\bigl(v^{\varepsilon }\bigr)
\bigr).
\end{equation}
In Proposition \ref{18} below we will establish that the map
$\bar{u}^{\varepsilon ,v^{\varepsilon }}$ is the unique solution of the
following stochastic controlled analogue
of equation (\ref{51})
%
\begin{align}\label{8}
\bar{u}^{\varepsilon ,v^{\varepsilon }}(t,x) & =  \frac{1}{h(\varepsilon)}
\int_{0}^{t}\int_{{0}}^{1}G_{t-s}(x,y)
\sigma \bigl(u^{0}(s,y) + \sqrt{ \varepsilon }h(\varepsilon )
\bar{u}^{\varepsilon ,v^{\varepsilon }}(s,y)\bigr)W(ds,dy)\nonumber\\
&\quad{}- {\sqrt{\varepsilon }h(\varepsilon )}\int_{0}^{t}
\int_{0}^{1} \partial _{y}G_{t-s}(x,y)
\bigl[\bar{u}^{\varepsilon ,v^{\varepsilon}}(s,y) \bigr]^{2}dyds\nonumber\\
&\quad{}- 2\int_{0}^{t}\int_{0}^{1} \partial _{y} G_{t-s}(x, y)
\bar{u} ^{\varepsilon ,v^{\varepsilon }}(s,y)u^{0}(s,y) dyds\nonumber\\
&\quad{}+ \int_{0}^{t}\int_{{0}}^{1}G_{t-s}(x,y)\sigma
\bigl(u^{0}(s,y) + \sqrt{ \varepsilon }h(\varepsilon )\bar{u}^{\varepsilon ,v^{\varepsilon }}(s,y)\bigr)
v ^{\varepsilon }(s,y)dyds.
\end{align}
%
We will call it the controlled process.
Moreover, for any $v\in S^{N}$, we associate to (\ref{8}) the following
skeleton zero-noise equation:
%
\begin{align}\label{17}
\bar{u}^{v}(t,x) &= - 2\int_{0}^{t}
\int_{0}^{1} \partial _{y}
G_{t-s}(x,y) \bar{u}^{v}(s,y)u^{0}(s, y) dyds\nonumber\\
&\quad{}+ \int_{0}^{t}\int_{{0}}^{1}G_{t-s}(x,y)
\sigma \bigl(u^{0}(s,y)\bigr)v(s,y)dyds.
\end{align}
Existence and uniqueness of the solution $\bar{u}^{v}$ for (\ref{17})
is obtained in Proposition \ref{22} below, and thereby, we define the
map
%
\begin{equation}
\label{G0} \mathcal{G}^{0}\bigl(u_{0},
\mathcal{I}(v)\bigr): = \bar{u}^{v}.
\end{equation}

With these notations in mind, the main result of this section is stated
in the following
%
\begin{thm}
\label{main}
Assume that $u_{0}$ is continuous, $\sigma $ is bounded and globally
Lipschitz\index{globally Lipschitz} and that (\ref{15}) holds. Then the family of processes
$\{\bar{u}^{\varepsilon };\, \varepsilon >0\}$ satisfies a LDP on the
space $C([0,T]; L^{2}([0,1]))$ with speed $\lambda ^{2}(\varepsilon )$
and rate function given by
%
\begin{equation}
I(f) = \inf \biggl\{ \frac{1}{2} \llVert v \rrVert _{\mathcal{H}_{T}}^{2},
\ v \in \mathcal{H}_{T}, \ \mathcal{G}^{0}
\bigl(u_{0}, \mathcal{I}(v)\bigr): = f \biggr\}.
\end{equation}
\end{thm}

\begin{rk}
Note that the conclusion of Theorem \ref{main} is still valid for \xch{a quite}{a quiet}
large class of SPDEs containing stochastic Burgers equation. Namely,
consider the following class of SPDEs\index{SPDEs} introduced by Gy\"{o}ngy in
\cite{gyongy1998existence}:
%
\begin{align}\label{remark019}
\dfrac{\partial u^{\varepsilon }}{\partial t}(t,x)
&=\dfrac{\partial^{2} }{\partial x^{2}}u^{\varepsilon }(t,x)
+ \dfrac{\partial }{\partial x}g \bigl(u^{\varepsilon }(t,x) \bigr)
+ f\bigl(u^{\varepsilon}(t,x) \bigr)\nonumber\\
&\quad{}+ \sqrt{\varepsilon }\sigma \bigl(u^{\varepsilon }(t,x)\bigr)
{\dot{W}}(t,x), \quad (t,x)\in [0,T]\times [0,1],
\end{align}
with Dirichlet\index{Dirichlet boundary conditions}'s boundary conditions $u^{\varepsilon }(t,0)=u^{\varepsilon
}(t,1)=0$ for $t\in [0,T]$, and the initial condition $u^{\varepsilon
}(0,x)=u_{0}(x)$ for $x\in [0,1]$.
Suitable conditions on the
coefficients $f$, $g$ and $\sigma $, for instance, the quadratic growth
assumption on the nonlinear coefficient $g$, bring us 
back to the case
of stochastic Burgers equation\index{stochastic Burgers equation} that we have considered in our paper.
Notice here that
papers closest to ours
are
two recent works by
S. Hu, R. Li and X. Wang \cite{hu2018central} and R. Zhang and
J. Xiong \cite{zhang2019semilinear}.
In particular, these authors established a moderate deviations \index{stochastic Burgers equation} principle for the class (\ref{remark019}).
%
We learned about these works after we
finished the first version of this paper.
\end{rk}

\subsection{Proof of the main result}

We basically follow the same idea as in
\cite{budhiraja2008large} and \cite{setayeshgar2014large}.
According to Proposition \ref{48}, it suffices to check that the
conditions \textbf{(A1)} and \textbf{(A2)} are fulfilled. For
\textbf{(A1)}, we will establish well-posedness, tightness\index{tightness} and
convergence\index{convergence} of controlled processes. The condition \textbf{(A2)}, which
gives that $I$ is a rate function, will follow from the continuity\index{continuity} of
the map $\mathcal{G}^{0}$ with respect to the weak topology.\index{weak topology}

The proof of \textbf{(A1)} will be done in several steps.

\medskip
\textbf{Step 1:} Existence and uniqueness of controlled and limiting processes.

\begin{proposition}
\label{18}
\label{12a}
Assume that $\sigma $ is bounded and globally Lipschitz,\index{globally Lipschitz} and that
(\ref{15}) holds. Then, the $\mathrm{L}^{2}([0,1])$-valued process
$\{\bar{u}^{\varepsilon ,v^{\varepsilon }}(t), t\in [0,T]\}$ defined
by (\ref{69}) is the unique solution of the equation (\ref{8}).
\end{proposition}

\begin{proof}
For $ v^{\varepsilon }\in {\mathcal{P}}_{2}^{N}$, set
\begin{align*}
dQ^{\varepsilon , v^{\varepsilon }}&:= \exp \Biggl\{-{\sqrt{h(\varepsilon )}}\int
_{0}^{t}\int_{0}^{1}v^{\varepsilon }(s,y)W(ds,dy)\\
&\quad{}-
\frac{1}{2}h( \varepsilon )\int_{0}^{t}
\int_{0}^{1}{v^{\varepsilon }}(s,y)^{2}dyds
\Biggr\}dP.
\end{align*}
Since $Q^{\varepsilon , v^{\varepsilon }}$ is defined through an
exponential martingale, it is a probability measure on $\varOmega $. And
thus, by the Girsanov theorem the process $\widetilde{W}$ defined by
\begin{equation*}
\widetilde{W}(dt,dx) = W(dt,dx) + h(\varepsilon )\int_{0}^{t}
\int_{0} ^{1}v^{\varepsilon }(s,y)dyds
\end{equation*}
is a space-time white noise under the probability measure $Q^{\varepsilon, v^{\varepsilon }}$.
Plugging $\widetilde{W}(dt,dx)$ in (\ref{8}) we
obtain (\ref{51}) with $\widetilde{W}(dt,dx)$ instead of ${W}(dt,dx)$.
Now, if ${u}$ denotes the unique solution of (\ref{51}) with
$\widetilde{W}(dt,dx)$ on the space $(\varOmega , \mathcal{F}, Q^{\varepsilon
, v^{\varepsilon }})$, then ${u}$ satisfies (\ref{8}), $Q^{\varepsilon
, v^{\varepsilon }}$ a.s. And hence by equivalence of probabilities,\index{probabilities}
${u}$ satisfies (\ref{8}), $P$ a.s.

For the uniqueness, if ${u}_{1}$ and ${u}_{2}$ are two solutions of
(\ref{8}) on $(\varOmega , \mathcal{F}, P)$, then ${u}_{1}$ and
${u}_{2}$ are solutions of (\ref{51}) governed by $\widetilde{W}(dt,dx)$
on $(\varOmega , \mathcal{F}, Q^{\varepsilon , v^{\varepsilon }})$. By the
uniqueness of the solution of (\ref{8}), we obtain ${u}_{1} = {u}_{2}$,
$Q^{\varepsilon , v^{\varepsilon }}$ a.s. And thus ${u}_{1} = {u}_{2}$,
$P$ a.s. by equivalence of probabilities.\index{probabilities}
\end{proof}

\begin{proposition}
\label{22}
Assume that $\sigma $ is bounded and globally Lipschitz.\index{globally Lipschitz} For any
$v\in {{S}}^{N}$, for some $N \in \mathbb{N}$, the equation (\ref{17})
admits a unique solution $\bar{u}^{v}$ belonging to $C([0,T];
\mathrm{L}^{2}([0,1]))$. Moreover, for any $q\geqslant 2$
%
\begin{equation}
\label{estm} \sup_{v\in {{S}}^{N}}\sup_{0\le t\le T}
\bigl\llVert \bar{u}^{v}(t,\cdot ) \bigr\rrVert _{2}^{q}<
\infty .
\end{equation}
\end{proposition}

\begin{proof}
The proof follows from a standard fixed point
argument, and for the convenience of the reader, we include it in the
Appendix.
\end{proof}

\textbf{Step 2:} Tightness\index{tightness} of the family $  (u^{\varepsilon
,v^{\varepsilon }}  )_{\varepsilon > 0}$ in $C([0,T];
L^{2}([0,1]))$.
\medskip

Let $  (v^{\varepsilon }  )_{\varepsilon }$ be a family of
elements from ${\mathcal{P}_{2}}^{N}$ such that $v^{\varepsilon }
\rightarrow v $ in distribution, as $S^{N}$-valued random elements, when
$\varepsilon \rightarrow 0$.

We have the following proposition.
%
\begin{proposition}
\label{25}Assume that $u_{0}$ is continuous, $\sigma $ is bounded and globally
Lipschitz,\index{globally Lipschitz} and that (\ref{15}) holds. Then $  (\bar{u}^{\varepsilon
,v^{\varepsilon }}  )_{\varepsilon }$ is tight in $C([0,T]; L
^{2}([0,1]))$.
\end{proposition}

\begin{proof}
Recall that
%
\begin{align}\label{decompo}
\bar{u}^{\varepsilon ,v^{\varepsilon }}(t,x)
& =  \frac{1}{h(\varepsilon)}\int_{0}^{t}\int_{{0}}^{1}G_{t-s}(x,y)
\sigma \bigl(u^{0}(s,y) + \sqrt{ \varepsilon }h(\varepsilon )
\bar{u}^{\varepsilon ,v^{\varepsilon }}(s,y) \bigr)W(ds,dy)
\nonumber\\[-2pt]
&\quad{}- {\sqrt{\varepsilon }h(\varepsilon )}\int_{0}^{t}\int_{0}^{1} \partial _{y}
G_{t-s}(x,y) \bigl[\bar{u}^{\varepsilon ,v^{\varepsilon}}(s,y) \bigr]^{2}dyds
\nonumber\\[-2pt]
&\quad{}- 2\int_{0}^{t}\int_{0}^{1}\partial _{y} G_{t-s}(x, y)
\bar{u} ^{\varepsilon ,v^{\varepsilon }}(s,y)u^{0}(s,y) dyds
\nonumber\\[-2pt]
&\quad{}+ \int_{0}^{t}\int_{{0}}^{1}G_{t-s}(x,y)\sigma
\bigl(u^{0}(s,y) + \sqrt{\varepsilon }h(\varepsilon )
\bar{u}^{\varepsilon ,v^{\varepsilon}}(s,y) \bigr)v^{\varepsilon }(s,y)dyds
\nonumber\\[-2pt]
& =: \sum_{i=1}^{4}I_{i}^{\varepsilon ,v^{\varepsilon }}(t,x),
\end{align}
where $I_{i}^{\varepsilon ,v^{\varepsilon }}(t,x)$, $i=1, 2, 3, 4$,
stands for the $i$\textsuperscript{\textit{th}} summand of the RHS of the above equation.

In view of (\ref{decompo}), in order to prove the claim of Proposition
\ref{25}, we will state and prove the next two lemmas which give the
tightness\index{tightness} of each summand $I_{i}^{\varepsilon ,v^{\varepsilon }}$,
$ i=1, 2, 3, 4$.
\end{proof}

We first consider the cases where $i=1$ and $i=4$. Using Theorem 4.10
of Chapter 2 in \cite{karatzas2012brownian}, the following lemma
states sufficient conditions for tightness.\index{tightness}

\begin{lem}
\label{tight1}
Assume the same conditions as in Proposition \ref{25}. For $i=1$ or
$4$, we have
%
\begin{equation}
\label{23} \lim_{\zeta \longrightarrow +\infty }\sup_{\varepsilon >0}P
\bigl( \bigl\llvert I _{i}^{\varepsilon ,v^{\varepsilon }}(t,x) \bigr\rrvert > \zeta
\bigr) = 0, \quad \text{for any } (t, x)\in [0,T]\times [0,1],
\end{equation}
and for any $\zeta >0$
%
\begin{equation}
\label{24} \hspace{-1cm} \quad \lim_{\delta \longrightarrow 0}\sup
_{\varepsilon >0}P \Bigl(\sup_{ \llvert t-t' \rrvert  +  \llvert x-y \rrvert \leqslant \delta } \bigl\llvert
I_{i}^{\varepsilon ,v^{
\varepsilon }}(t,x) - I_{i}^{\varepsilon ,v^{\varepsilon }}
\bigl(t',y\bigr) \bigr\rrvert > \zeta \Bigr) = 0.
\end{equation}
In particular, the families $  (I_{1}^{\varepsilon ,v^{\varepsilon
}}  )_{\varepsilon }$ and $  (I_{4}^{\varepsilon ,v^{\varepsilon
}}  )_{\varepsilon }$ are tight in $C([0,T]; L^{2}([0,1]))$.
\end{lem}

\begin{proof}
Let $x, y \in [0,1]$ and $t, t' \in [0,T]$ such that
$t' \leqslant t$. To prove (\ref{23}) and (\ref{24}), it is enough to
exhibit upper bounds for the square moments of $I_{i}^{\varepsilon ,v
^{\varepsilon }}(t,x)$ and $I_{i}^{\varepsilon ,v^{\varepsilon }}(t,x)
- I_{i}^{\varepsilon ,v^{\varepsilon }}(t',y)$ for $i= 1$ and $i=4$.

Using the Burkholder--Davis--Gundy{} inequality, the boundedness\index{boundedness} of
$\sigma $, Lemma \ref{20} and the condition (\ref{15}) we infer that
%
\begin{align}\label{100}
&\mathbf{E} \bigl( \bigl\llvert I_{1}^{\varepsilon ,v^{\varepsilon }}(t,x)
\bigr\rrvert ^{2} \bigr)\nonumber\\
&\quad  \leqslant  c. h^{-2}(
\varepsilon ). \mathbf{E} \int_{0}^{t}\int
_{0}^{1}G_{t-s}^{2}(x,y)
\sigma ^{2} \bigl(u^{0}(s,y) + \sqrt{ \varepsilon }h(
\varepsilon )\bar{u}^{\varepsilon ,v^{\varepsilon }}(s,y) \bigr)dyds
\nonumber\\
&\quad  \leqslant  c.\int_{0}^{t}\int
_{0}^{1}G_{t-s}^{2}(x,y)dyds,
\end{align}
which is finite. On the other hand, the same arguments as above yield
%
\begin{align}\label{101}
&  \mathbf{E} \bigl( \bigl\llvert I_{1}^{\varepsilon ,v^{\varepsilon }}(t,x)
- I_{1}^{\varepsilon ,v^{\varepsilon }}\bigl(t',y\bigr) \bigr\rrvert
^{2} \bigr)
\nonumber
\\[-2pt]
&\quad  = h^{-2}(\varepsilon ).\mathbf{E}\Biggl\{ \int
_{0}^{t'}\int_{0}
^{1}\bigl[G_{t-s}(x,z) - G_{t'-s}(y,z)\bigr]\nonumber\\
&\qquad \qquad \qquad\ \times\sigma \bigl(u^{0}(s,y) + \sqrt{ \varepsilon }h(\varepsilon )
\bar{u}^{\varepsilon ,v^{\varepsilon }}(s,y) \bigr)W(ds,dz)
\nonumber
\\[-2pt]
& \qquad \qquad \qquad\ +  \int_{t'}^{t}
\int_{0}^{1}\xch{G_{t-s}}{[G_{t-s}}(y,z) \sigma
\bigl(u^{0}(s,y) + \sqrt{\varepsilon }h(\varepsilon )
\bar{u}^{\varepsilon ,v^{\varepsilon }}(s,y) \bigr)W(ds,dz)\Biggr\} ^{2}
\nonumber
\\[-2pt]
&\quad  \leqslant c \Biggl\{\int_{0}^{t'}\int
_{0}^{1}\bigl[G_{t-s}(x,z) -
G_{t'-s}(y,z)\bigr]^{2}dzds + \int_{t'}^{t}
\int_{0}^{1}G_{t-s}^{2}(y,z)dzds
\Biggr\}
\nonumber
\\[-2pt]
&\quad  \leqslant c \bigl( \bigl\llvert t-t' \bigr\rrvert
^{\frac{1}{2}} + \bigl\llVert x- x' \bigr\rrVert
^{\frac{1}{2}} \bigr).
\end{align}
Therefore, (\ref{23}) and (\ref{24}) hold by (\ref{100}) and
(\ref{101}), respectively.

To deal with $  (I_{4}^{\varepsilon ,v^{\varepsilon }}  )
_{\varepsilon }$, we use the Cauchy--Schwarz inequality and Lemma \ref{20}
to write
%
\begin{align}
\label{102} \mathbf{E} \bigl( \bigl\llvert I_{4}^{\varepsilon ,v^{\varepsilon }}(t,x)
\bigr\rrvert ^{2} \bigr) & \leqslant  c \mathbf{E} \Biggl(\int
_{0}^{t}\int_{0}
^{1} \bigl\llvert G_{t-s}(x,y) v^{\varepsilon }(s,y)
\bigr\rrvert dyds \Biggr)^{2}
\nonumber
\\
& \leqslant  c \bigl\llVert v^{\varepsilon } \bigr\rrVert _{\mathcal{H}_{T}}^{2}.
\int_{0} ^{t}\int_{0}^{1}G_{t-s}^{2}(x,y)dyds
\nonumber
\\
& \leqslant  c(N),
\end{align}
where $c(N)$ is a constant depending on $N$. Similarly,
%
\begin{align}\label{103}
& \mathbf{E} \bigl( \bigl\llvert I_{4}^{\varepsilon ,v^{\varepsilon }}(t,x)
- I_{4}^{\varepsilon ,v^{\varepsilon }}\bigl(t',y\bigr) \bigr\rrvert
^{2} \bigr)
\nonumber\\
&\quad = \mathbf{E}\Biggl\{ \int_{0}^{t'}
\int_{0}^{1}\bigl[G_{t-s}(x,z) - G
_{t'-s}(y,z)\bigr] \sigma \bigl(u^{\varepsilon ,v^{\varepsilon }}(s,z) \bigr)v
^{\varepsilon }(s,y)dzds
\nonumber\\
& \qquad +  \int_{t'}^{t}\int
_{0}^{1}\xch{G_{t-s}}{[G_{t-s}}(y,z) \sigma
\bigl(u^{\varepsilon ,v^{\varepsilon }}(s,z) \bigr)v^{\varepsilon }(s,y)dzds\Biggr\}
^{2}
\nonumber
\\
&\quad  \leqslant c \Biggl\{\int_{0}^{t'}\int
_{0}^{1}\bigl[G_{t-s}(x,z) -
G_{t'-s}(y,z)\bigr]^{2}dzds + \int_{t'}^{t}
\int_{0}^{1}G_{t-s}^{2}(y,z)dzds
\Biggr\}
\nonumber
\\
&\quad  \leqslant c \bigl( \bigl\llvert t-t' \bigr\rrvert
^{\frac{1}{2}} + \bigl\llVert x- x' \bigr\rrVert
^{\frac{1}{2}} \bigr).
\end{align}
Therefore, (\ref{23}) and (\ref{24}) hold by (\ref{102}) and
(\ref{103}), respectively.
\end{proof}

For the tightness\index{tightness} of $  (I_{2}^{\varepsilon ,v^{\varepsilon }}  )
_{\varepsilon }$, we follow an idea introduced in
\cite{gyongy1998existence} which is essentially based on Lemma
\ref{19} in the Appendix. More precisely, we state the following

\begin{lem}
\label{tight2}
Assume the same conditions as in Proposition \ref{25}. Then, the
families $(I_{2}^{\varepsilon ,v^{\varepsilon }})_{\varepsilon }$ and
$(I_{3}^{\varepsilon ,v^{\varepsilon }})_{\varepsilon }$ are uniformly
tight in $C([0,T]; L^{2}([0,1]))$.
\end{lem}

\begin{proof} The proof of the tightness\index{tightness} of $  (I_{3}^{\varepsilon
,v^{\varepsilon }}  )_{\varepsilon }$ will be omitted since it can
be done similarly to this of $  (I_{2}^{\varepsilon ,v^{\varepsilon
}}  )_{\varepsilon }$.

To show the tightness\index{tightness} of $  (I_{2}^{\varepsilon ,v^{\varepsilon }}  )
_{\varepsilon }$, we will apply Lemma \ref{19} with $q=1$, $\rho = 2$
and $\zeta _{\varepsilon }(t,\cdot ): = \sqrt{\varepsilon }h(\varepsilon
) (\bar{u}^{\varepsilon ,v^{\varepsilon }})^{2}(t,\cdot )$. Set
\begin{equation*}
\theta _{\varepsilon } := \sqrt{\varepsilon }h(\varepsilon ) \sup
_{0\leqslant t \leqslant T} \bigl\llVert \bigl(\bar{u}^{\varepsilon ,v^{\varepsilon
}}
\bigr)^{2}(t,\cdot ) \bigr\rrVert _{1} = \sqrt{
\varepsilon }h(\varepsilon ) \sup_{0\leqslant t \leqslant T} \bigl\llVert
\bar{u}^{\varepsilon ,v^{\varepsilon
}}(t,\cdot ) \bigr\rrVert _{2}^{2}.
\end{equation*}
According to Lemma \ref{19}, it suffices to show that $(
\theta _{\varepsilon })_{\varepsilon }$ is bounded in probability. i.e.
%
\begin{eqnarray}
\label{110} \lim_{c \longrightarrow +\infty }\sup_{\varepsilon > 0}
\mathbb{P} (\theta _{\varepsilon } \geqslant c ) = 0.
\end{eqnarray}
Taking into account the condition (\ref{15}), there exists $
\varepsilon _{0} > 0$ such that $\sqrt{\varepsilon }h(\varepsilon )
\leqslant 1$ for all $\varepsilon \leqslant \varepsilon _{0}$.
Consequently
\begin{align*}
\sup_{\varepsilon \leqslant \varepsilon _{0}}\mathbb{P} (\theta _{\varepsilon
} \geqslant c )
& =  \sup_{\varepsilon \leqslant \varepsilon
_{0}}\mathbb{P} \biggl(\sup_{0\leqslant t \leqslant T}
\bigl\llVert \bar{u}^{
\varepsilon ,v^{\varepsilon }}(t,\cdot ) \bigr\rrVert _{2}^{2}
\geqslant \frac{c}{\sqrt{
\varepsilon }h(\varepsilon )} \biggr)
\nonumber
\\
& \leqslant  \sup_{\varepsilon \leqslant \varepsilon _{0}}\mathbb{P} \Bigl(\sup
_{0\leqslant t \leqslant T} \bigl\llVert \bar{u}^{\varepsilon ,v^{
\varepsilon }}(t,\cdot ) \bigr
\rrVert _{2}^{2} \geqslant c \Bigr).
\end{align*}
Then, to prove (\ref{110}), it is enough to show that
%
\begin{eqnarray}
\lim_{c \longrightarrow +\infty } \sup_{\varepsilon \leqslant \varepsilon _{0}}\mathbb{P} \Bigl(
\sup_{0
\leqslant t \leqslant T} \bigl\llVert \bar{u}^{\varepsilon ,v^{\varepsilon }}(t, \cdot )
\bigr\rrVert _{2} \geqslant c \Bigr) = 0.
\end{eqnarray}
For this purpose, returning to (\ref{decompo}) we note that
$\bar{u}^{\varepsilon ,v^{\varepsilon }}$ corresponds to the following
SPDE
%
\begin{align}\label{120}
\dfrac{\partial \bar{u}^{\varepsilon ,v^{\varepsilon }}}{\partial t}(t,x)
& = \Delta \bar{u}^{\varepsilon ,v^{\varepsilon }}(t,x) + \dfrac{
\partial g_{\varepsilon }}{\partial x} \bigl(t, x, u^{\varepsilon , v
^{\varepsilon }}(t,x) \bigr) +
f_{\varepsilon } \bigl(t, x, u^{
\varepsilon , v^{\varepsilon }}(t,x) \bigr)
\nonumber
\\
&\quad{}+ \sigma _{\varepsilon } \bigl(t, x, \bar{u}^{\varepsilon ,v^{
\varepsilon }}(t,x)
\bigr){\dot{W}}(t,x),
\end{align}
where
\begin{align*}
g_{\varepsilon }(t, x, r) &:= -\sqrt{\varepsilon }h(\varepsilon )r
^{2} - 2ru^{0}(t,x),\\
f_{\varepsilon }(t, x, r) &:= \sigma   (u^{0}(t,x) + \sqrt{\varepsilon }h(\varepsilon )r  )v^{\varepsilon }(t,x)
\quad \mbox{and}\\
\sigma _{\varepsilon }(t, x, r) &:= \frac{1}{h(\varepsilon
)}\sigma   (u^{0}(t,x) +\sqrt{\varepsilon }h(\varepsilon )r  ).
\end{align*}

According to Theorem 2.1 in \cite{gyongy1998existence}, the
continuity\index{continuity} of the initial condition $u_{0}$ implies the continuity\index{continuity} of
the solution $u^{0}$ of the equation (\ref{2}) on the compact set
$[0,T]\times [0,1]$. Consequently, $u^{0}$ is bounded.

This fact combined with the condition (\ref{15}) allows us to 
consider the
function $g_{\varepsilon }$ as a sum of two functions $g_{\varepsilon
}^{1}$ and $g_{\varepsilon }^{2}$ satisfying major quadratic and linear 
conditions respectively, uniformly in $\varepsilon $ being less than certain~$
\varepsilon _{0}$.

Using again the condition (\ref{15}) and the hypotheses on the function
$\sigma $, we see that $\sigma _{\varepsilon }$ is bounded and globally
Lipschitzian,\index{globally Lipschitz} uniformly in $\varepsilon $ being less than certain
$\varepsilon _{0}$.

Thus, the equation (\ref{120})
belongs to the class of semi-linear
SPDE studied in \cite{gyongy1998existence}, and for which the
existence and uniqueness of the solution $\bar{u}^{\varepsilon ,v^{
\varepsilon }}$ is showed by an approximation procedure. This procedure
is to define a sequence of truncated equations, and to establish
existence and some convergence results for the corresponding sequence
of solutions $  (\bar{u}^{\varepsilon ,v^{\varepsilon }}_{n}  )
_{n}$, see \cite{gyongy1998existence,foondun2017large,setayeshgar2014large}. In fact, in
the course of the proof of Theorem 2.1 in
\cite{gyongy1998existence} it was shown that
%
\begin{equation}
\label{130} \lim_{c \longrightarrow \infty } \sup_{0 < \varepsilon \leqslant \varepsilon _{0}}
\mathbb{P} \biggl(\sup_{0\leqslant t \leqslant T} \bigl\llVert
\bar{u}^{\varepsilon ,v^{\varepsilon }} _{n}(t, \cdot ) \bigr\rrVert
_{2} \geqslant \frac{c}{2} \biggr) =0,
\end{equation}
and that $  (\bar{u}^{\varepsilon ,v^{\varepsilon }}_{n}  )
_{n}$ converges in probability in $C([0,T]; L^{2}([0,1]))$ to the
solution $\bar{u}^{\varepsilon ,v^{\varepsilon }}$ of (\ref{decompo}).

Now, observe that
\begin{align*}
&\sup_{0 < \varepsilon \leqslant \varepsilon _{0}}\mathbb{P} \Bigl\{ \Bigl(\sup
_{0\leqslant t \leqslant T} \bigl\llVert \bar{u}^{\varepsilon ,v^{
\varepsilon }}(t,\cdot ) \bigr
\rrVert _{2} \Bigr) \geqslant c \Bigr\}\\
&\quad  \leqslant  \sup
_{0 < \varepsilon \leqslant \varepsilon _{0}}\mathbb{P} \biggl(\sup_{0\leqslant t \leqslant T} \bigl
\llVert \bar{u}^{\varepsilon ,v^{\varepsilon }}(t, \cdot ) - \bar{u}^{\varepsilon ,v^{\varepsilon }}_{n}(t,
\cdot ) \bigr\rrVert _{2} \geqslant \frac{c}{2} \biggr)
\\
&\qquad{}+ \sup_{0 < \varepsilon \leqslant \varepsilon _{0}}\mathbb{P} \biggl(\sup_{0\leqslant t \leqslant T}
\bigl\llVert \bar{u}^{\varepsilon ,v^{\varepsilon }} _{n}(t, \cdot ) \bigr\rrVert
_{2} \geqslant \frac{c}{2} \biggr).
\end{align*}
Then, as $c$ tends to infinity, the estimate (\ref{130}) yields
\begin{align*}
&\lim_{c \longrightarrow +\infty } \sup_{0 < \varepsilon \leqslant \varepsilon _{0}}\mathbb{P} \Bigl\{
\Bigl(\sup_{0\leqslant t \leqslant T} \bigl\llVert \bar{u}^{\varepsilon ,v^{\varepsilon }}(t,
\cdot ) \bigr\rrVert _{2} \Bigr) \geqslant c \Bigr\}\\
&\quad  \leqslant \lim
_{c \longrightarrow +\infty } \sup_{0 < \varepsilon \leqslant \varepsilon _{0}} \mathbb{P} \biggl(\sup
_{0\leqslant t \leqslant T} \bigl\llVert \bar{u}^{\varepsilon ,v^{\varepsilon }}(t, \cdot ) -
\bar{u}^{\varepsilon ,v^{\varepsilon }}_{n}(t, \cdot ) \bigr\rrVert
_{2} \geqslant \frac{c}{2} \biggr).
\end{align*}
By letting $n$ tend to infinity and using the convergence\index{convergence} in
probability of $\bar{u}^{\varepsilon ,v^{\varepsilon }}_{n}$ to
$\bar{u}^{\varepsilon ,v^{\varepsilon }}$ we get
\begin{eqnarray*}
\lim_{c \longrightarrow +\infty } \sup_{0 < \varepsilon \leqslant \varepsilon _{0}}\mathbb{P} \Bigl\{
\Bigl(\sup_{0\leqslant t
\leqslant T} \bigl\llVert \bar{u}^{\varepsilon ,v^{\varepsilon }}(t,
\cdot ) \bigr\rrVert _{2} \Bigr) \geqslant c \Bigr\} = 0.
\end{eqnarray*}
Hence, by applying Lemma \ref{19} we obtain the tightness property for
$  (I_{2}^{\varepsilon ,v^{\varepsilon }}  )_{\varepsilon }$.

\medskip
\textbf{Step 3:} Convergence\index{convergence} to the limit equation.
\medskip

Having shown the tightness\index{tightness} of each $I_{i}^{\varepsilon ,v^{\varepsilon
}}$ for $i=1, 2, 3, 4$, by Prohorov's theorem, we can extract a
subsequence, that we continue to denote by $\varepsilon $, and along
which each of these processes and $\bar{u}^{\varepsilon ,v^{\varepsilon
}}$
converge in distribution (as $S^{N}$-valued random elements) in
$C([0,T]; L^{2}([0,1]))$ to limits denoted respectively by $I_{i}^{0,v}$
for $i=1, 2, 3, 4$, and $\bar{u}^{0,v}$. We will show that
\begin{align*}
I_{1}^{0,v} & =  0,
\\
I_{2}^{0,v} & =  0,
\\
I_{3}^{0,v} & =  - 2\int_{0}^{t}
\int_{0}^{1} \partial _{y}
G_{t-s}(x, y) {u}^{0, v}(s,y)u^{0}(s, y) dyds,
\\
I_{4}^{0,v} & =  \int_{0}^{t}
\int_{{0}}^{1}G_{t-s}(x,y)\sigma
\bigl(u^{0}(s,y)\bigr)v(s,y)dyds,
\end{align*}
and the proof will be completed by the uniqueness result given in
Proposition \ref{22}.

For $i=1$, Lemma 3 in \cite{budhiraja2008large} ensures the
convergence\index{convergence} of $  (I_{1}^{\varepsilon , v^{\varepsilon }}  )
_{\varepsilon }$ to $0$ in probability in
$C  ([0,T]\times [0,1]  )$. And, while the convergence\index{convergence} in
probability in $C([0, T]\times [0, 1])$ implies the one in $C([0,T]; L
^{2}([0,1]))$, hence $  (I_{1}^{\varepsilon , v^{\varepsilon }}  )
_{\varepsilon }$ converges to $0$ in probability in $C([0, T];
L^{2}([0,1]))$ too.

To handle the convergence\index{convergence} of each of the other terms, we invoke the
Skorohod representation theorem and
assume the almost sure convergence
on a larger common probability space.

For $i=2$, applying Lemma \ref{gyongylemma} with $\rho = 2$ and
$\lambda =1$, we deduce there exists a constant $c > 0$ such that
\begin{align*}
\bigl\llVert I_{2}^{\varepsilon , v^{\varepsilon }}(t, \cdot ) \bigr\rrVert
_{2}  \leqslant c\sqrt{\varepsilon }h(\varepsilon ) \int
_{0}^{t}(t-s)^{-
\frac{3}{4}}\|
\bar{u}^{\varepsilon ,v^{\varepsilon }}(s, \cdot \xch{)}{))}
\| _{2}^{2}ds.
\end{align*}
And since $  (\bar{u}^{\varepsilon ,v^{\varepsilon }}  )_{
\varepsilon }$ converges a.s. in $C([0,T]; L^{2}([0,1]))$ to
$\bar{u}^{0,v}$, then there exists $\varepsilon _{0} > 0$ small enough
such that
%
\begin{eqnarray}
\label{90} \sup_{\varepsilon \in ]0, \varepsilon _{0}]}\sup_{s\in [0, T]}
\bigl\llVert \bar{u}^{\varepsilon ,v^{\varepsilon }}(s, \cdot ) \bigr\rrVert _{2} <
\infty , \quad \text{a.s.}
\end{eqnarray}
%
Further, there exists a constant $c > 0$ such that for all $ 0 < \varepsilon
\leqslant \varepsilon _{0}$
\begin{align*}
\sup_{t\in [0,T]} \bigl\llVert I_{2}^{\varepsilon , v^{\varepsilon }}(t,
\cdot ) \bigr\rrVert _{2}  \leqslant  c\sqrt{\varepsilon }h(
\varepsilon\xch{ ),}{ ).} \quad \text{a.s.}
\end{align*}
Thus, $  (I_{2}^{\varepsilon , v^{\varepsilon }}  )_{\varepsilon
}$ converges a.s. to $0$ in $C([0, T]; L^{2}([0,1]))$ as $\varepsilon
$ tends to $0$.

For $i=3$, let $\tilde{I}_{3}^{0,v}$ denote the RHS term
of ${I}_{3}^{0,v}$. Applying again Lemma \ref{gyongylemma} and
the Cauchy--Schwarz inequality,
we conclude that there exists a constant $c > 0$ such that
\begin{align*}
\bigl\llVert I_{3}^{\varepsilon , v^{\varepsilon }}(t, \cdot ) -
\tilde{I}_{3}^{0,v}(t, \cdot ) \bigr\rrVert
_{2} & \leqslant  c \int_{0}
^{t}(t-s)^{-\frac{3}{4}} \bigl\llVert \bigl(\bar{u}^{\varepsilon ,v^{
\varepsilon }}(s,
\cdot ) - \bar{u}^{0,v}(s,\cdot ) \bigr)u^{0}(s, \cdot )
\bigr\rrVert _{1}ds
\\
& \leqslant  c\int_{0}^{t}(t-s)^{-\frac{3}{4}}
\bigl\llVert \bar{u}^{
\varepsilon ,v^{\varepsilon }}(s,\cdot ) - \bar{u}^{0,v}(s,
\cdot ) \bigr\rrVert _{2} \bigl\llVert u^{0}(s, \cdot )
\bigr\rrVert _{2}ds.
\end{align*}

Using the estimation (\ref{6bis}) and the boundedness\index{boundedness} of $\bar{u}^{
\varepsilon ,v^{\varepsilon }}$ and $\bar{u}^{0,v}$ in $C([0,T];\break  L
^{2}([0,1]))$, we get
\begin{align*}
&\bigl\llVert I_{3}^{\varepsilon , v^{\varepsilon }}(t, \cdot ) -
\tilde{I}_{3}^{0,v}(t, \cdot ) \bigr\rrVert
_{2}\\
&\quad  \leqslant  c \sup_{s\in [0,T]} \bigl\llVert
\bar{u}^{\varepsilon ,v^{\varepsilon }}(s,\cdot ) - \bar{u}^{0,v}(s, \cdot ) \bigr
\rrVert _{2}\sup_{s\in [0,T]} \bigl\llVert
u^{0}(s, \cdot ) \bigr\rrVert _{2} \int
_{0}^{t}(t-s)^{-\frac{3}{4}}ds
\\
&\quad  \leqslant  c \sup_{s\in [0,T]} \bigl\llVert
\bar{u}^{\varepsilon ,v^{\varepsilon
}}(s,\cdot ) - \bar{u}^{0,v}(s, \cdot ) \bigr
\rrVert _{2}.
\end{align*}
Again, since $  (\bar{u}^{\varepsilon ,v^{\varepsilon }}  )
_{\varepsilon }$ converges a.s. in $C([0,T]; L^{2}([0,1]))$ to
$\bar{u}^{0,v}$, we obtain the a.s. convergence\index{convergence} of $I_{3}^{\varepsilon
, v^{\varepsilon }}$ to $\tilde{I}_{3}^{0,v}$ in $C([0,T];
L^{2}([0,1]))$. And by the uniqueness of the limit and the continuity\index{continuity}
of $\tilde{I}_{3}^{0,v}$, we conclude that ${I}_{3}^{0,v} = \tilde{I}
_{3}^{0,v}$.

Concerning $i=4$, let $\tilde{I}_{4}^{0,v}$ denote the RHS term of ${I}_{4}^{0,v}$. We have
\begin{align*}
& I_{4}^{\varepsilon , v^{\varepsilon }}(t, \cdot ) - \tilde{I}_{4}
^{0,v}(t, \cdot )
\\
&\quad  = \int_{0}^{t}\int_{0}^{1}G_{t-s}(x,y)
\bigl[\sigma \bigl(u^{0}(s,y) + \sqrt{\varepsilon }h(\varepsilon )
\bar{u}^{\varepsilon ,v^{\varepsilon
}}(s,y)\bigr)v^{\varepsilon }(s,y)\\
&\qquad\qquad\qquad\qquad\qquad  - \sigma
\bigl(u^{0}(s,y)\bigr)v(s,y) \bigr]dyds
\\
& \quad  = \int_{0}^{t}\int_{0}^{1}G_{t-s}(x,y)
\bigl[\sigma \bigl(u^{0}(s,y) + \sqrt{\varepsilon }h(\varepsilon )
\bar{u}^{\varepsilon ,v^{\varepsilon
}}(s,y)\bigr)\\
&\qquad\qquad\qquad\qquad\qquad  - \sigma \bigl(u^{0}(s,y)\bigr)
\bigr]v^{\varepsilon }(s,y)dyds
\\
& \qquad + \int_{0}^{t}\int
_{0}^{1}G_{t-s}(x,y)
\bigl[v^{\varepsilon
}(s,y) - v(s,y) \bigr]\sigma \bigl(u^{0}(s,y)
\bigr)dyds
\\
& \quad  =: J_{4, 1}^{\varepsilon }(t,x) + J_{4, 2}^{\varepsilon }(t,x).
\end{align*}
Then,
\begin{equation*}
\bigl\llVert I_{4}^{\varepsilon , v^{\varepsilon }}(t, \cdot ) -
\tilde{I}_{4} ^{(0,v)}(t, \cdot ) \bigr\rrVert
_{2} \leqslant \bigl\llVert J_{4, 1}^{\varepsilon }(t,
\cdot ) \bigr\rrVert _{2} + \bigl\llVert J_{4, 2}^{\varepsilon }(t,
\cdot ) \bigr\rrVert _{2}.
\end{equation*}
For $J_{4, 1}^{\varepsilon }$, we use Lemma \ref{gyongylemma}, the Lipschitz
condition on $\sigma $ and the Cauchy--Schwarz inequality to obtain
\begin{align*}
&\bigl\llVert J_{4, 1}^{\varepsilon }(t,\cdot ) \bigr\rrVert
_{2}\\
 &\quad  \leqslant  c \int_{0}
^{t}(t-s)^{-\frac{3}{4}} \bigl\llVert \bigl(\sigma
\bigl(u^{0}(s, \cdot ) + \sqrt{ \varepsilon }h(\varepsilon )
\bar{u}^{\varepsilon ,v^{\varepsilon }}(s, \cdot )\bigr) - \sigma \bigl(u^{0}(s,
\cdot )\bigr) \bigr)v^{\varepsilon }(s, \cdot ) \bigr\rrVert _{1}ds
\nonumber
\\
&\quad  \leqslant  c \sqrt{\varepsilon }h(\varepsilon ) \int_{0}^{t}(t-s)^{-
\frac{3}{4}}
\bigl\llVert \bar{u}^{\varepsilon ,v^{\varepsilon }}(s, \cdot ) \bigr\rrVert _{2}
\bigl\llVert v^{\varepsilon }(s, \cdot ) \bigr\rrVert _{2}ds.
\end{align*}
Since $(v^{\varepsilon })\subset \mathcal{P}_{2}^{N}$,
the estimation (\ref{90}) implies that there exists a constant $c$ depending
on $N$ such that for all $ 0 < \varepsilon \leqslant \varepsilon _{0}$
\begin{align*}
\sup_{t\in [0,T]} \bigl\llVert J_{4, 1}^{\varepsilon }(t,
\cdot ) \bigr\rrVert _{2}  \leqslant  c \sqrt{\varepsilon }h(
\varepsilon ), \quad \text{a.s.}
\end{align*}
Therefore, $J_{4, 1}^{\varepsilon }$ converges to $0$ in $C([0,T]; L
^{2}[0,1])$ as $\varepsilon $ goes to $0$.

The proof of the convergence\index{convergence} of $J_{4, 2}^{\varepsilon }$ to $0$ in
$C([0,T]; L^{2}[0,1])$ as $\varepsilon $ goes to $0$ will be omitted
since it can be treated similarly to the case of the family
$\{K_{n},  n\geqslant 1\}$ defined below by (\ref{83}).

Consequently, $I_{4}^{\varepsilon , v^{\varepsilon }}$ converges to
$\tilde{I}_{4}^{0,v}$ in $C([0,T]; L^{2}([0,1]))$, and by the uniqueness
of the limit and the continuity\index{continuity} of $\tilde{I}_{4}^{0,v}$, we conclude
that ${I}_{4}^{0,v} = \tilde{I}_{4}^{0,v}$.

Thus, by the convergence\index{convergence} of both the process $  (\bar{u}^{\varepsilon
,v^{\varepsilon }}  )_{\varepsilon }$ and each term $I_{i}^{
\varepsilon , ,v^{\varepsilon }}$\break  for $i=1, 2, 3, 4 $ along a
subsequence, and
by the uniqueness of the solution\break  of
the equation (\ref{17}), we conclude that the condition \textbf{(A1)}
in Proposition \ref{48}\break  holds.
\medskip

Now, let us prove the condition \textbf{(A2)}. As it was mentioned
before, it suffices to check the continuity\index{continuity} of the map $\mathcal{G}
^{0}: \mathcal{E}_{0}\times S^{N} \longrightarrow C([0,T];
L^{2}([0,1]))$ with respect to the weak topology.\index{weak topology} Let $v$, $(v_{n})
\subset S^{N}$ such that for any $g\in \mathcal{H}_{T}$,
\begin{equation*}
\lim_{n\longrightarrow +\infty } \langle v - v_{n} , g \rangle
_{\mathcal{H}_{T}} =0.
\end{equation*}
We claim that
%
\begin{equation}
\label{claim1} \lim_{n \longrightarrow +\infty }\sup_{t\in [0,T]}
\bigl\llVert u^{v_{n}}(t) - u ^{v}(t) \bigr\rrVert
_{2}=0.
\end{equation}

Let $(t,x)\in [0,T]\times [0,1]$. The equation (\ref{17}) implies
%
\begin{align}
\bar{u}^{v_{n}}(t,x) - \bar{u}^{v}(t,x) & =  - 2\int
_{0}^{t}\int_{0}
^{1} \partial _{y} G_{t-s}(x,y)u^{0}(s,
y) \bigl(\bar{u}^{v_{n}}(s,y) - \bar{u}^{v}(s,y) \bigr)dyds
\nonumber
\\
&\quad{}+ \int_{0}^{t}\int_{{0}}^{1}G_{t-s}(x,y)
\sigma \bigl(u^{0}(s,y)\bigr) \bigl(v_{n}(s,y) -v(s,y)
\bigr)dyds.
\end{align}
Hence,
%
\begin{align}\label{64}
&\bigl\llVert \bar{u}^{v_{n}}(t, \cdot ) -
\bar{u}^{v}(t, \cdot ) \bigr\rrVert _{2}\nonumber\\
&\quad  \leqslant c
\Biggl\{ \Biggl\llVert \int_{0}^{t}\int
_{0}^{1} \partial _{y}
G_{t-s}( \cdot ,y)u^{0}(s, y) \bigl(\bar{u}^{v_{n}}(s,y)
- \bar{u}^{v}(s,y) \bigr)dyds \Biggr\rrVert _{2}
\nonumber\\
&\qquad{}+ \Biggl\llVert \int_{0}^{t}\int
_{{0}}^{1}G_{t-s}(\cdot ,y) \sigma
\bigl(u^{0}(s,y)\bigr) \bigl(v_{n}(s,y) -v(s,y)
\bigr)dyds \Biggr\rrVert _{2}\Biggr\}.
\end{align}

On one hand, using Lemma \ref{gyongylemma}, the Cauchy--Schwarz inequality
and estimation (\ref{6bis}) we get
%
\begin{align}
\label{66}
& \Biggl\llVert \int_{0}^{t}
\int_{0}^{1} \partial _{y}
G_{t-s}(\cdot ,y)u ^{0}(s, y) \bigl(\bar{u}^{v_{n}}(s,y)
- \bar{u}^{v}(s,y) \bigr)dyds \Biggr\rrVert _{2}
\nonumber
\\
&\quad  \leqslant c \int_{0}^{t}(t-s)^{-3/4}
\bigl\llVert u^{0}(s,\cdot ) \bigl(\bar{u} ^{v_{n}}(s,\cdot
) - \bar{u}^{v}(s,\cdot ) \bigr) \bigr\rrVert _{1}ds
\nonumber
\\
&\quad  \leqslant c \int_{0}^{t}(t-s)^{-3/4}
\bigl\llVert u^{0}(s,\cdot ) \bigr\rrVert _{2} \bigl
\llVert \bar{u}^{v_{n}}(s,\cdot ) - \bar{u}^{v}(s,\cdot )
\bigr\rrVert _{2}ds
\nonumber
\\
&\quad  \leqslant c \int_{0}^{t}(t-s)^{-3/4}
\bigl\llVert u^{0}(s,\cdot ) \bigr\rrVert _{2} \bigl
\llVert \bar{u}^{v_{n}}(s,\cdot ) - \bar{u}^{v}(s,\cdot )
\bigr\rrVert _{2}ds
\nonumber
\\
&\quad  \leqslant c \int_{0}^{t}(t-s)^{-3/4}
\sup_{s\in [0,T]} \bigl\llVert u^{0}(s, \cdot ) \bigr
\rrVert _{2} \bigl\llVert \bar{u}^{v_{n}}(s,\cdot ) -
\bar{u}^{v}(s,\cdot ) \bigr\rrVert _{2}ds
\nonumber
\\
&\quad  \leqslant c \int_{0}^{t}(t-s)^{-3/4}
\bigl\llVert \bar{u}^{v_{n}}(s,\cdot ) - \bar{u}^{v}(s,
\cdot ) \bigr\rrVert _{2}ds.
\end{align}

On the other hand, in order to handle the second term in the right hand
side 
of (\ref{64}), we define, for any
$(t,x)\in [0,T]\times [0,1]$, the sequence
%
\begin{eqnarray}
\label{83} K_{n}(t,x)  := \int_{0}^{t}
\int_{{0}}^{1}G_{t-s}(x,y)\sigma
\bigl(u^{0}(s,y)\bigr) \bigl(v _{n}(s,y) -v(s,y)
\bigr)dyds,
\end{eqnarray}
whose properties are given in Lemma \ref{80} in the Appendix. Then,
summing up (\ref{64})--(\ref{66}), we obtain for any $0\leqslant
t \leqslant T$
%
\begin{align}
\bigl\llVert \bar{u}^{v_{n}}(t, \cdot ) - \bar{u}^{v}(t,
\cdot ) \bigr\rrVert _{2} \leqslant c \bigl\llVert
K_{n}(t) \bigr\rrVert _{2} + c \int
_{0}^{t}(t-s)^{-3/4} \bigl\llVert
\bar{u}^{v_{n}}(s, \cdot ) - \bar{u}^{v}(s,\cdot ) \bigr
\rrVert _{2}ds.
\end{align}
Applying Gronwall's lemma, we get the estimate
%
\begin{equation}
\label{67} \sup_{t\in [0,T]} \bigl\llVert
\bar{u}^{v_{n}}(t, \cdot ) - \bar{u}^{v}(t, \cdot ) \bigr
\rrVert _{2} \leqslant c\sup_{t\in [0,T]} \bigl\llVert
K_{n}(t, \cdot ) \bigr\rrVert _{2},
\end{equation}
which implies together with (\ref{82}) the claim (\ref{claim1}), and
henceforth the condition \textbf{(A2)} holds.

Finally, the proof of Theorem \ref{main} is completed since conditions
of Proposition \ref{48} are fulfilled.
\end{proof}

\section{Toward a central limit\index{central limit theorem} theorem}\label{sec4}

Many results on central limit theorem\index{central limit theorem} has been recently established for
various kinds of parabolic SPDEs\index{parabolic SPDEs} under strong assumptions on the drift
coefficient. More specifically, under the linear growth condition, the
differentiability and the global Liptschitz condition on both the drift
coefficient and its derivative, some central limit theorems\index{central limit theorem} have been
established in \cite{wang2015moderate,yang2016moderate}. And while these conditions are not all
fulfilled for the stochastic Burgers equation,\index{stochastic Burgers equation} it is not surprising that
classical tools do not apply to establish a central limit\index{central limit theorem} theorem.
Nevertheless, we will prove in this section two first-step results
toward a central limit\index{central limit theorem} theorem. More specifically, the uniform
boundedness and the convergence\index{convergence} of $u^{\varepsilon }$ to $u^{0}$ in
$\mathrm{L}^{q}(\varOmega ; C([0,T]; \mathrm{L}^{2}([0,1])))$ for
$q\geqslant 2$. We hope that our current estimates could be helpful for
future works in this direction.

We begin with the following result.
%
\begin{proposition}
\label{53}
Assume that $\sigma $ is bounded and globally Lipschitz.\index{globally Lipschitz} Then for all
$q\geqslant 2$, we have
%
\begin{equation}
\label{34} \sup_{\varepsilon \in ]0,1]}\mathbf{E} \Bigl(\sup
_{t\in [0,T]} \bigl\llVert u^{
\varepsilon }(t,\cdot ) \bigr\rrVert
_{2}^{q} \Bigr) < \infty .
\end{equation}
\end{proposition}

\begin{proof}
We will use similar arguments as in Cardon-Weber and
Millet \cite{cardon2001support} and Gy\"{o}ngy
\cite{gyongy1998existence}. For $0 < \varepsilon \leqslant 1$, set
\begin{equation*}
\eta _{\varepsilon }(t,x): = \sqrt{\varepsilon } \int_{0}^{t}
\int_{0}^{1}G_{t-s}(x,y)\sigma
\bigl(u^{\varepsilon }(s,y)\bigr)W(dy,ds),
\end{equation*}
and
\begin{align*}
\vartheta ^{\varepsilon }(t,x) & :=  u^{\varepsilon }(t,x) - \eta
_{\varepsilon }(t,x)
\\
& =  \int_{0}^{1}G_{t}(x,y)u_{0}(y)dy
\,{-}\, \int_{0}^{t}\!\int_{0}^{1}
\partial _{y} G_{t-s}(x,y) \bigl(u^{\varepsilon }(s,y)
\bigr)^{2}dyds
\\
& =  \int_{0}^{1}G_{t}(x,y)u_{0}(y)dy
\,{-}\, \int_{0}^{t}\!\int_{0}^{1}
\partial _{y} G_{t-s}(x,y) \bigl(\vartheta
^{\varepsilon }(s,y) + \eta _{\varepsilon }(s,y)\bigr)^{2}dyds.
\end{align*}
Then, $\vartheta ^{\varepsilon }$ is a solution of the 
equation
%
\begin{equation}
\label{58} \dfrac{\partial \vartheta ^{\varepsilon }}{\partial t}(t,x)= \Delta \vartheta
^{\varepsilon }(t,x) + \dfrac{\partial }{\partial x} \bigl(\vartheta ^{\varepsilon }(t,x)
+ \eta _{\varepsilon }(t,x) \bigr)^{2}, \quad  (t,x)\in [0,T]\times
[0, 1],
\end{equation}
with Dirichlet\index{Dirichlet boundary conditions}'s boundary conditions and initial condition $
\vartheta ^{\varepsilon }(0,x)=u_{0}(x)$.

Since $\sigma \circ u^{\varepsilon }$ is bounded uniformly in
$\varepsilon $, arguing as in the proof of Theorem $2.1$ in
{\cite{gyongy1998existence}, page 286}, by
the Garsia--Rodemich--Rumsey lemma, one can deduce that
\begin{eqnarray*}
\sup_{\varepsilon }\mathbf{E} \Bigl(\sup_{0 \leqslant t \leqslant T}
\sup_{0
\leqslant x \leqslant 1} \bigl\llvert \tilde{\eta }_{\varepsilon }(t,x)
\bigr\rrvert ^{q} \Bigr) <\infty ,
\end{eqnarray*}
where $\tilde{\eta }_{\varepsilon }(t,x):=\frac{1}{\sqrt{\varepsilon
}}\eta _{\varepsilon }(t,x)$. Consequently, there exists a universal
constant $C(q)$ depending only on $q$ such that
%
\begin{eqnarray}
\label{ineq:BDG} \mathbf{E} \Bigl(\sup_{0 \leqslant t \leqslant T} \sup
_{0
\leqslant x \leqslant 1} \bigl\llvert \eta _{\varepsilon }(t,x) \bigr\rrvert
^{q} \Bigr)  \leqslant  C(p) \varepsilon ^{q/2}.
\end{eqnarray}
In particular, the random variable $ \bar{\eta }_{\varepsilon } : =
\sup_{0 \leqslant t \leqslant T}\sup_{0 \leqslant x \leqslant 1} \llvert \eta
_{\varepsilon }(t,x) \rrvert $ is well defined a.s.

Moreover, using the SPDE (\ref{58}) satisfied by $\vartheta ^{\varepsilon
}$ and following the same arguments as in the proof of Theorem 2.1 in
\cite{gyongy1998existence}, we deduce the existence of a constant
$c$ independent of $\varepsilon $ and $\omega $ (see
\cite{gyongy1998existence} pages 286--289) such that
%
\begin{eqnarray}
\label{57} \sup_{0 \leqslant t \leqslant T} \bigl\llVert \vartheta
^{\varepsilon }\bigl(t,^{
\cdot }\bigr) \bigr\rrVert _{2}^{2}
\leqslant \llVert u_{0} \rrVert _{2}^{2} + cT
\bigl(1 + \bar{ \eta }_{\varepsilon }^{4} \bigr)e^{  (cT(1+
\bar{\eta }_{\varepsilon
}^{2})  )}.
\end{eqnarray}
Consequently, for any $q \geqslant 2$
\begin{align*}
\sup_{0 \leqslant t \leqslant T} \bigl\llVert u^{\varepsilon }(t, \cdot ) \bigr
\rrVert _{2}^{q} & =  \sup_{0 \leqslant t \leqslant T}
\bigl\llVert \vartheta ^{\varepsilon
}(t, \cdot ) + \eta _{\varepsilon }(t,
\cdot ) \bigr\rrVert _{2}^{q}
\\
& \leqslant  2^{q-1} \Bigl(\sup_{0 \leqslant t \leqslant T} \bigl
\llVert \vartheta ^{\varepsilon }(t, \cdot ) \bigr\rrVert _{2}^{q}
+ \sup_{0 \leqslant t \leqslant T} \bigl\llVert \eta _{\varepsilon }(t, \cdot )
\bigr\rrVert _{2}^{q} \Bigr)
\\
& \leqslant 2^{q-1} \Biggl( \llVert u_{0} \rrVert
_{2}^{q} + cT \bigl(1 + \bar{ \eta }_{\varepsilon }^{2q}
\bigr)e^{  (cT(1+
\bar{\eta }_{\varepsilon
}^{2})  )}\\
&\quad + \sup_{0 \leqslant t \leqslant T} \Biggl(\int
_{0}^{1} \bigl\llvert \eta _{\varepsilon }(t,
x) \bigr\rrvert ^{2}dx \Biggr)^{q/2} \Biggr)
\\
& \leqslant  2^{q-1} \bigl( \llVert u_{0} \rrVert
_{2}^{q} + cT \bigl(1 + \bar{ \eta }_{\varepsilon }^{2q}
\bigr)e^{  (cT(1+
\bar{\eta }_{\varepsilon
}^{2})  )} + \bar{\eta }_{\varepsilon }^{q} \bigr)
\\
& \leqslant  c \bigl( \llVert u_{0} \rrVert _{2}^{q}
+ cT \bigl(1 + \bar{\eta } _{\varepsilon }^{2q}
\bigr)e^{  (cT(1+
\bar{\eta }_{\varepsilon
}^{2})  )} \bigr).
\end{align*}

Hence, to prove (\ref{34}) it suffices to show that
%
\begin{equation}
\label{62} \sup_{\varepsilon \in ]0,1]} \mathbf{E} \bigl( \bigl(1 +
\bar{\eta } _{\varepsilon }^{2q} \bigr)e^{cT(1+
\bar{\eta }_{\varepsilon }^{2})} \bigr) \quad  \mbox{ is finite.}
\end{equation}

For this purpose, note first that
\begin{equation*}
\sup_{0\leqslant s \leqslant T}\sup_{0\leqslant x \leqslant 1} \bigl\llvert
\sqrt{ \varepsilon }\sigma \bigl(u^{\varepsilon }(s,x)\bigr) \bigr\rrvert
\leqslant \sqrt{\varepsilon } \llVert \sigma \rrVert _{\infty }, \quad
\mbox{where }  \llVert \sigma \rrVert _{\infty }:= \sup
_{x\in \mathbb{R}} \bigl\llvert \sigma (x) \bigr\rrvert .
\end{equation*}
Thus, by Lemma \ref{56}, there exist two positive constants
$C_{1}$ and $C_{2}$, independent of $\varepsilon $, such that for any
$M\geqslant C_{1}\|\sigma \|_{\infty }$
%
\begin{equation}
P (\bar{\eta }_{\varepsilon } \geqslant M ) \leqslant C _{1}
\llVert \sigma \rrVert _{\infty }\exp \biggl(-\frac{M^{2}}{\varepsilon C_{2}
  (1+T^{\frac{1}{8}}  )} \biggr).
\end{equation}
Setting $\varphi (x): = (1 + x^{2q})e^{cT  (1+x^{2}  )}$, which
is a positive, continuous and increasing function on $[0, +\infty [$,
we get for any $A\geqslant C_{1}\|\sigma \|_{\infty }$
%
\begin{align}
\mathbf{E}\bigl(\varphi (\bar{\eta }_{\varepsilon })\bigr) &= \int
_{0}^{+
\infty }P \bigl(\varphi (\bar{\eta
}_{\varepsilon }) > x \bigr)dx
\nonumber
\\
&= \int_{0}^{A}P (\bar{\eta }_{\varepsilon }
> x )\varphi '(x)dx + \int_{A}^{+\infty }P
(\bar{\eta }_{\varepsilon } > x ) \varphi '(x)dx
\nonumber
\\
& \leqslant  \varphi (A) + cC_{1} \llVert \sigma \rrVert
_{\infty }\int_{A}^{+
\infty } \bigl(1 +
x^{2q+1} \bigr)\exp \biggl(cT x^{2}-\frac{x^{2}}{
\varepsilon C_{2}  (1+T^{\frac{1}{8}}  )}
\biggr) dx
\nonumber
\\
& \leqslant  \varphi (A) + cC_{1} \llVert \sigma \rrVert
_{\infty }\int_{A}^{+
\infty } \bigl(1 +
x^{2q+1} \bigr)\exp \biggl(cT x^{2}-\frac{x^{2}}{C
_{2}  (1+T^{\frac{1}{8}}  )}
\biggr)dx,
\nonumber
\end{align}
where the last integral is finite provided that $cTC_{2}  (1+T
^{\frac{1}{8}}  ) < 1$. This implies that there exists $T_{0} > 0$,
independent of $u_{0}$ and $\varepsilon $, such that (\ref{62}) holds
for $0 < T \leqslant T_{0}$. Using (\ref{57}), and 
iterating the
procedure finitely many times we conclude the proof.
\end{proof}

Now, we can announce and state the following proposition.
%
\begin{proposition}
\label{68}
Assume that $\sigma $ is bounded and globally Lipschitz.\index{globally Lipschitz} Then, for all
$q\geqslant 2$, we have
%
\begin{equation}
\label{33} \lim_{\varepsilon \longrightarrow 0}\mathbf{E} \Bigl(\sup
_{t\in [0,T]} \bigl\llVert u^{\varepsilon }(t,\cdot ) -
u^{0}(t,\cdot ) \bigr\rrVert _{2}^{q}
\Bigr)=\xch{0.}{0,}
\end{equation}
\end{proposition}

\begin{proof} We will use a localization argument. For $0\leqslant
t \leqslant T$, $\varepsilon \in ]0,1]$ and $M >0$, set
%
\begin{equation}
\label{39} \varOmega _{\varepsilon }^{M}(t) := \Bigl\{ w\in
\varOmega : \sup_{s\in [0,t]} \bigl\llVert u ^{\varepsilon }(s)
\bigr\rrVert _{2} \vee \sup_{s\in [0,t]} \bigl\llVert
u^{0}(s) \bigr\rrVert _{2} \leqslant M \Bigr\}.
\end{equation}
We have
%
\begin{align}
u^{\varepsilon }(t, x) - u^{0}(t, x) & =  \sqrt{\varepsilon }\int
_{0}^{t}\int_{0}^{1}G_{t-s}(x,y)
\sigma \bigl(u^{\varepsilon }(s,y)\bigr)W(ds,dy)
\nonumber
\\
& \quad{}- \int_{0}^{t}\int_{0}^{1}
\partial _{y} G_{t-s}(x,y) \bigl(\bigl(u^{
\varepsilon }(s,y)
\bigr)^{2} - \bigl(u^{0}(s,y)\bigr)^{2}
\bigr)dyds
\nonumber
\\
& :=  \eta ^{\varepsilon }(t,x) + I^{\varepsilon }(t,x).
\end{align}
Then, for any $q\geqslant 2$,
%
\begin{eqnarray}
\bigl\llVert u^{\varepsilon }(t, \cdot ) - u^{0}(t, \cdot ) \bigr
\rrVert _{2}^{q} \leqslant 2^{q-1} \bigl( \bigl
\llVert \eta ^{\varepsilon }(t,\cdot ) \bigr\rrVert ^{q} + \bigl
\llVert I^{\varepsilon
}(t,\cdot ) \bigr\rrVert ^{q} \bigr).
\label{29}
\end{eqnarray}
For $\eta ^{\varepsilon }(t,\cdot )$, by the H\"{o}lder inequality we have
\begin{align*}
\mathbf{E} \Bigl(\sup_{0\leqslant s \leqslant t} \bigl\llVert \eta
^{\varepsilon }(s, \cdot ) \bigr\rrVert ^{q} \Bigr) &\leqslant
\mathbf{E} \Biggl(\sup_{0\leqslant s
\leqslant t}\int_{0}^{1}
\bigl\llvert \eta ^{\varepsilon }(s,x) \bigr\rrvert ^{q}dx \Biggr)
\\
& \leqslant  \int_{0}^{1}\mathbf{E} \Bigl(
\sup_{0\leqslant s \leqslant
t} \bigl\llvert \eta ^{\varepsilon }(s,x) \bigr
\rrvert ^{q} \Bigr)dx
\\
&\leqslant  \mathbf{E} \Bigl(\sup_{0\leqslant x\leqslant 1} \sup
_{0\leqslant s \leqslant t} \bigl\llvert \eta ^{\varepsilon }(s,x) \bigr\rrvert
^{q} \Bigr)
\\
&\leqslant  C(q)\varepsilon ^{q/2},
\end{align*}
where the last inequality follows from (\ref{ineq:BDG}).

For $I^{\varepsilon }(t,\cdot )$, according to Lemma
\ref{gyongylemma} in the Appendix with $\rho = 2$ and $\lambda =1$, we
have
%
\begin{align}
\bigl\llVert I^{\varepsilon }(t,\cdot ) \bigr\rrVert _{2}
\leqslant  c\int_{0}^{t}(t-s)^{-
\frac{3}{4}}
\bigl\llVert \bigl(u^{\varepsilon }(s,\cdot ) - u^{0}(s,\cdot )\bigr)
\bigl(u^{
\varepsilon }(s,\cdot ) + u^{0}(s,\cdot )\bigr) \bigr
\rrVert _{1}\xch{ds,}{ds.} \label{32}
\end{align}
and using the following form of H\"{o}lder's inequality
\begin{align*}
 \Biggl\llvert \int_{0}^{t}f(s)g(s)ds\Biggr \rrvert ^{q} \leqslant   \Biggl(\int_{0}
^{t}\bigl|f(s)\bigr|ds  \Biggr)^{q-1}\int_{0}^{t}\bigl|f(s)\bigr|\bigl|g(s)\bigr|^{q}ds,
\end{align*}
with $f(s):= (t-s)^{-\frac{3}{4}}$ and $g(s): = \|(u^{\varepsilon }(s,
\cdot ) - u^{0}(s,\cdot ))(u^{\varepsilon }(s,\cdot ) + u^{0}(s,
\cdot ))\|_{1}$, we get
%
\begin{align}
\label{45} \bigl\llVert I^{\varepsilon }(t,\cdot ) \bigr\rrVert
_{2}^{q}  \leqslant c\int_{0}^{t}(t-s)^{-
\frac{3}{4}}
\bigl\llVert \bigl(u^{\varepsilon }(s,\cdot ) - u^{0}(s,\cdot )
\bigr) \bigl(u^{
\varepsilon }(s,\cdot ) + u^{0}(s,\cdot )\bigr)
\bigr\rrVert _{1}^{q}ds.
\end{align}

Now, taking the supremum up to time $t\in [0, T]$, and setting $
\varPhi (s): = \|(u^{\varepsilon }(s,\cdot ) - u^{0}(s,\cdot ))(u^{
\varepsilon }(s,\cdot ) + u^{0}(s,\cdot ))\|_{1}^{q} $, and
$\varPsi (s): = \sup_{0\leqslant r \leqslant s}\varPhi (r)$, (\ref{45}) implies
%
\begin{align}
\label{43} \sup_{0\leqslant s \leqslant t} \bigl\llVert
I^{\varepsilon }(s,\cdot ) \bigr\rrVert _{2}^{q} &
\leqslant  c\sup_{0\leqslant s \leqslant t}\int_{0}^{s}(s-r)^{-
\frac{3}{4}}
\varPhi (r)\xch{dr}{dr.}
\nonumber
\\
& \leqslant  c\sup_{0\leqslant s \leqslant t}\int_{0}^{s}(s-r)^{-
\frac{3}{4}}
\sup_{0\leqslant r' \leqslant r}\varPhi \bigl(r'\bigr)\xch{dr}{dr.}
\nonumber
\\
& \leqslant  c\sup_{0\leqslant s \leqslant t}\int_{0}^{s}(s-r)^{-
\frac{3}{4}}
\varPsi (r)\xch{dr}{dr.}
\nonumber
\\
& = c\sup_{0\leqslant s \leqslant t}\int_{0}^{s}r^{-\frac{3}{4}}
\varPsi (s-r)dr.
\end{align}
Since
\begin{equation*}
\varPsi (s-r) = \sup_{0\leqslant r' \leqslant s-r}\varPhi (r') \leqslant
\sup_{0\leqslant r' \leqslant t-r}\varPhi (r') = \varPsi (t-r),
\end{equation*}
then
%
\begin{align}
\sup_{0\leqslant s \leqslant t} \bigl\llVert I^{\varepsilon }(s,\cdot ) \bigr
\rrVert _{2}^{q} & \leqslant  c\sup_{0\leqslant s \leqslant t}
\int_{0}^{s}r^{-
\frac{3}{4}} \varPsi (t-r)dr
\nonumber
\\
& =  c\int_{0}^{t}r^{-\frac{3}{4}} \varPsi
(t-r)dr
\nonumber
\\
& =  c\int_{0}^{t}(t-r)^{-\frac{3}{4}}
\varPsi (r)dr.
\nonumber
\end{align}
Introducing the expectation on $\varOmega _{\varepsilon }^{M}(t)$ and taking
into account the facts that $\varOmega _{\varepsilon }^{M}(t)\in
\mathcal{F}_{t}$ and $\varOmega _{\varepsilon }^{M}(t) \subset
\varOmega _{\varepsilon }^{M}(s) $ for $0\leqslant s \leqslant t$, we get
%
\begin{eqnarray}
\label{63} \mathbf{E} \Bigl(\mathbf{1}_{\varOmega _{\varepsilon }^{M}(t)} \sup
_{0\leqslant s
\leqslant t} \bigl\llVert I^{\varepsilon }(s,\cdot ) \bigr\rrVert
_{2}^{q} \Bigr)  \leqslant  c\int_{0}^{t}(t-s)^{-\frac{3}{4}}
\mathbf{E} \bigl(\mathbf{1}_{
\varOmega _{\varepsilon }^{M}(s)}\varPsi (s) \bigr)ds.
\end{eqnarray}

Notice that
%
\begin{align}
\mathbf{1}_{\varOmega _{\varepsilon }^{M}(s)}\varPsi (s) & \leqslant  \mathbf{1}_{\varOmega _{\varepsilon }^{M}(s)}
\sup_{0\leqslant r\leqslant
s} \bigl\llVert \bigl(u^{\varepsilon }(r,\cdot ) -
u^{0}(r,\cdot )\bigr) \bigl(u^{\varepsilon }(r, \cdot ) +
u^{0}(r,\cdot )\bigr) \bigr\rrVert _{1}^{q}
\nonumber
\\
& \leqslant  \mathbf{1}_{\varOmega _{\varepsilon }^{M}(s)} \sup_{0\leqslant r\leqslant s} \bigl
\llVert u^{\varepsilon }(r,\cdot ) - u^{0}(r, \cdot ) \bigr\rrVert
_{2}^{q} \bigl\llVert u^{\varepsilon }(r,\cdot ) +
u^{0}(r,\cdot ) \bigr\rrVert _{2} ^{q}
\nonumber
\\
& \leqslant  \mathbf{1}_{\varOmega _{\varepsilon }^{M}(s)} \sup_{0\leqslant r\leqslant s} \bigl
\llVert u^{\varepsilon }(r,\cdot ) - u^{0}(r, \cdot ) \bigr\rrVert
_{2}^{q} \bigl( \bigl\llVert u^{\varepsilon }(r,\cdot )
\bigr\rrVert _{2}^{q} + \bigl\llVert u ^{0}(r,
\cdot ) \bigr\rrVert _{2}^{q} \xch{\bigr)}{\bigr).}
\nonumber
\\
& \leqslant  2M^{q}\mathbf{1}_{\varOmega _{\varepsilon }^{M}(s)} \sup
_{0\leqslant r\leqslant s} \bigl\llVert u^{\varepsilon }(r,\cdot ) -
u^{0}(r, \cdot ) \bigr\rrVert _{2}^{q}.
\nonumber
\end{align}

This, together with (\ref{63}), gives
%
\begin{align}
&\mathbf{E} \Bigl(\mathbf{1}_{\varOmega _{\varepsilon }^{M}(t)} \sup_{0\leqslant s\leqslant t}
\bigl\llVert I^{\varepsilon }(s,\cdot ) \bigr\rrVert _{2}^{q}
\Bigr)\nonumber\\
&\quad   \leqslant  2cM^{q}\int_{0}^{t}(t-s)^{-\frac{3}{4}}
\mathbf{E} \Bigl(\mathbf{1} _{\varOmega _{\varepsilon }^{M}(s)}\sup_{0\leqslant r\leqslant s}
\bigl\llVert u^{
\varepsilon }(r,\cdot ) - u^{0}(r,\cdot ) \bigr
\rrVert _{2}^{q} \Bigr)ds.
\label{27}
\end{align}

Combining (\ref{29})--(\ref{27}) we get for any $0 \leqslant t \leqslant
T$
%
\begin{align}
& \mathbf{E} \Bigl(\mathbf{1}_{\varOmega _{\varepsilon }^{M}(t)} \sup_{0\leqslant s \leqslant t}
\bigl\llVert u^{\varepsilon }(s,\cdot ) - u^{0}(s, \cdot ) \bigr
\rrVert _{2}^{q} \Bigr)
\nonumber
\\
&\quad \leqslant c \Biggl[\varepsilon ^{q/2} + 2M^{q} \int
_{0}^{t}(t-s)^{-
\frac{3}{4}} \mathbf{E} \Bigl(
\mathbf{1}_{\varOmega _{\varepsilon }^{M}(s)} \sup_{0\leqslant r\leqslant s} \bigl\llVert
u^{\varepsilon }(r,\cdot ) - u^{0}(r, \cdot ) \bigr\rrVert
_{2}^{q} \Bigr)ds \Biggr].
\end{align}

Using Gronwall's lemma we deduce that, for all $t\in [0, T]$,
%
\begin{equation}
\mathbf{E} \Bigl(\mathbf{1}_{\varOmega _{\varepsilon }^{M}(t)} \sup_{0\leqslant s \leqslant t}
\bigl\llVert u^{\varepsilon }(s,\cdot ) - u^{0}(s, \cdot ) \bigr
\rrVert _{2}^{q} \Bigr) \leqslant c \varepsilon
^{q/2}e^{2cM^{q}}.
\end{equation}

Therefore, for any fixed $M >0$ we have
%
\begin{align}\label{37}
&\mathbf{E} \Bigl(\sup_{0\leqslant t \leqslant T} \bigl\llVert
u^{\varepsilon }(t, \cdot ) - u^{0}(t,\cdot ) \bigr\rrVert
_{2}^{q} \Bigr)
\nonumber
\\
&\quad  = \mathbf{E} \Bigl(\mathbf{1}_{\varOmega _{\varepsilon }^{M}(T)} \sup_{0\leqslant t \leqslant T}
\bigl\llVert u^{\varepsilon }(t,\cdot ) - u^{0}(t, \cdot ) \bigr
\rrVert _{2}^{q} \Bigr)\nonumber\\
&\qquad{}  + \mathbf{E} \Bigl(
\mathbf{1}_{\varOmega
\setminus \varOmega _{\varepsilon }^{M}(T)}\sup_{0\leqslant t \leqslant T} \bigl\llVert
u^{\varepsilon }(t,\cdot ) - u^{0}(t,\cdot ) \bigr\rrVert
_{2}^{q} \Bigr)
\nonumber
\\
&\quad  \leqslant c \varepsilon ^{q/2}e^{2cM^{q}} + \bigl(P \bigl(
\varOmega \setminus \varOmega _{\varepsilon }^{M}(T) \bigr)
\bigr)^{1/2} \Bigl(\mathbf{E} \Bigl(\sup_{0\leqslant t
\leqslant T}
\bigl\llVert u^{\varepsilon }(t,\cdot ) - u ^{0}(t,\cdot ) \bigr
\rrVert _{2}^{2q} \Bigr) \Bigr)^{1/2}.
\nonumber
\end{align}
To deal with the second term of the last inequality, on one hand,
estimations (\ref{6bis}) and (\ref{34}) imply that there exists $c>0$
such that
%
\begin{equation}
\label{59} \sup_{\varepsilon \in ]0,1]}\mathbf{E} \Bigl(\sup
_{0\leqslant t
\leqslant
T} \bigl\llVert u^{\varepsilon }(t,\cdot ) -
u^{0}(t,\cdot ) \bigr\rrVert _{2}^{q} \Bigr)
< c.
\end{equation}
On the other hand, by the Markov inequality and using again the estimations
(\ref{6bis}) and (\ref{34}) we have
%
\begin{align}
P \bigl(\varOmega \setminus \varOmega _{\varepsilon }^{M}(T)
\bigr) & \leqslant  P \Bigl(\sup_{t\in [0,T]} \bigl\llVert
u^{\varepsilon }(t,\cdot ) \bigr\rrVert _{2}^{q} >
M^{q} \Bigr) + P \Bigl(\sup_{t\in [0,T]} \bigl\llVert
u^{0}(t,\cdot ) \bigr\rrVert _{2}^{q} >
M^{q} \Bigr)
\nonumber
\\
& \leqslant  \frac{\mathbf{E}  (\sup_{t\in [0,T]} \llVert u^{\varepsilon
}(t,\cdot ) \rrVert _{2}^{q}  )}{M^{q}} + \frac{\mathbf{E}  (\sup_{t
\in [0,T]} \llVert u^{0}(t,\cdot ) \rrVert _{2}^{q}  )}{M^{q}}
\nonumber
\\
& \leqslant  \frac{\sup_{\varepsilon \in ]0, 1]}\mathbf{E}  (\sup_{t\in [0,T]} \llVert u^{\varepsilon }(t,\cdot ) \rrVert _{2}^{q}  )}{M^{q}} + \frac{\sup_{t\in [0,T]} \llVert u^{0}(t,\cdot ) \rrVert _{2}^{q}}{M^{q}}
\nonumber
\\
& \leqslant  \frac{c}{M^{q}}.
\end{align}
Then
%
\begin{eqnarray}
\mathbf{E} \Bigl(\sup_{0\leqslant t \leqslant T} \bigl\llVert
u^{\varepsilon }(t, \cdot ) - u^{0}(t,\cdot ) \bigr\rrVert
_{2}^{q} \Bigr) \leqslant c \varepsilon
^{q/2}e^{2cM^{q}} + \frac{c}{{M}^{q/2}}.
\end{eqnarray}

Letting $\varepsilon $ tends to zero and taking into account the fact
that $\varepsilon $ and $M$ are independent, we obtain
\begin{equation*}
\limsup_{\varepsilon \longrightarrow 0}\mathbf{E} \Bigl(\sup_{0\leqslant
t \leqslant T}
\bigl\llVert u^{\varepsilon }(t,\cdot ) - u^{0}(t,\cdot ) \bigr
\rrVert _{2} ^{q} \Bigr) \leqslant \frac{c}{{M}^{q/2}}.
\end{equation*}

Finally, since $M$ is arbitrary, we conclude that (\ref{33}) holds.
\end{proof}

\begin{appendix}
\section*{Appendix}
\renewcommand{\thethm}{4.\arabic{thm}}
\setcounter{thm}{0}

This section contains some technical results needed in the proof of the
main theorem of the paper.

First, we recall the following result proved in Lemma 3.1 in
\cite{gyongy1998existence}.

For $H(t,s;x,y):= G(t-s, x, y)$ or $H(t,s;x,y):= (\partial /\partial
_{y})G(t-s, x, y)$, where $0\leqslant s \leqslant t \leqslant T$ and
$x,y \in [0,1]$, define the linear operator $J$ by
\begin{equation*}
J(v) (t,x) := \int_{0}^{t}\int
_{0}^{1}H(r, t; x,y)v(r,y)dydr, \quad t \in [0,T], \ x\in [0,1],
\end{equation*}
for every $v\in \mathrm{L}^{\infty }  ([0,T], L^{1}([0,1])  )$.
%
\begin{lem}
\label{gyongylemma}
Let $\rho > 1$, $\lambda \in [1, \rho [$ and set $\kappa := 1 + \frac{1}{
\rho } - \frac{1}{\lambda }$. Then, $J$ is a bounded linear operator
from $\mathrm{L}^{\gamma }  ([0,T], L^{\lambda }([0,1])  )$
into $C  ([0,T], L^{\rho }([0,1])  )$ for $\gamma > 2\kappa
^{-1}$. Moreover, there exists a positive constant $c$ such that for all
$t\in [0, T]$,
%
\begin{equation}
\bigl\llVert J(v) (t, \cdot ) \bigr\rrVert _{\rho } \leqslant c \int
_{0}^{t}(t-r)^{\frac{
\kappa }{2}
- 1} \bigl\llVert v(r,
\cdot ) \bigr\rrVert _{\lambda }dr.
\end{equation}
\end{lem}

The following lemma is a consequence of Lemma \xch{3.1}{3.1.} in
\cite{chenal1997uniform}, its proof is omitted.
%
\begin{lem}
\label{56}
Let $\mathcal{F}_{t}= \sigma (W(s,x); 0\leqslant s\leqslant t; 0
\leqslant x \leqslant 1)$ and let $Z: \varOmega \times [0,T]\times [0,1]
\longrightarrow \mathbb{R}$ be a $\mathcal{F}_{t}$-predictable process
such that
\begin{equation*}
\sup_{0\leqslant s \leqslant T}\sup_{0\leqslant y \leqslant
1}|Z(s,y)| \leqslant \rho .
\end{equation*}

Set $I(t,x) := \int_{0}^{t}\int_{0}^{1}G_{t-u}(y,z)Z(u,z)W(du,dz)$.
Then, there exist positive constants $C_{1}$ and $C_{2}$ such that for
$M>C_{1}\rho $,
%
\begin{equation}
\label{55} P \Bigl(\sup_{0\leqslant s \leqslant T}\sup
_{0\leqslant y
\leqslant 1} \bigl\llvert I(s,y) \bigr\rrvert \geqslant M \Bigr)
\leqslant C_{1}\exp \biggl(-\frac{M
^{2}}{\rho ^{2}C_{2}  (1+T^{\frac{1}{8}}  )} \biggr).
\end{equation}
\end{lem}

\begin{proof}[Proof of Proposition \ref{22}]
To use a fixed point argument, we consider, for any given $\mathrm{L}^{2}([0,1])$-valued
function $\{w(t), t\in [0, T]\}$, the following operator
\begin{align*}
(\mathcal{A}w) (t,x) &:= - 2\int_{0}^{t}\int
_{0}^{1} \partial _{y} G
_{t-s}(x,y)w(s,y)u^{0}(s, y) dyds\\ &\quad{}+ \int
_{0}^{t}\int_{{0}}^{1}G_{t-s}(x,y)
\sigma \bigl(u^{0}(s,y)\bigr)v(s,y)dyds.
\end{align*}
We are going to prove that $\mathcal{A}$ is a contraction operator
on the Banach space $\mathbb{H}$ of $\mathrm{L}^{2}([0,1])$-valued
functions $\{w(t), t\in [0, T]\}$ such that $u(0) = 0$ equipped with
the norm
%
\begin{equation}
\llVert w \rrVert : = \int_{0}^{T}e^{-\lambda t}
\bigl\llVert w(t, \cdot ) \bigr\rrVert _{2}^{2}dt, \quad
\text{where} \ \lambda >0 \ \text{will be fixed later}.
\end{equation}

{Step 1.} Let $t\in [0, T]$. We first prove that if $w$ satisfies
$
\sup_{0\leqslant s \leqslant t}\|w(s,\cdot )\|_{2}^{q} < \infty $ then
$\mathcal{A}w$ satisfies also this estimate. By Lemma
\ref{gyongylemma}, the Cauchy--Schwarz inequality and the hypothesis on
$w$ we have
%
\begin{align}
\label{36} \Bigl(\sup_{0\leqslant s \leqslant t} \bigl\llVert
\mathcal{A}w(t,\cdot ) \bigr\rrVert _{2} ^{q} \Bigr) &
\leqslant c \Biggl[ 1 + \int_{0}^{t}(t-s)^{
\frac{-3}{4}}
\Bigl(\sup_{0\leqslant r
\leqslant s} \bigl\llVert w(r,\cdot )u^{0}(r,
\cdot ) \bigr\rrVert _{1}^{q} \Bigr)ds \Biggr]
\nonumber
\\
& \leqslant  c \Biggl[ 1 + \int_{0}^{t}(t-s)^{\frac{-3}{4}}
\Bigl(\sup_{0\leqslant r
\leqslant s} \bigl\llVert w(r,\cdot ) \bigr\rrVert
_{2}^{q} \bigl\llVert u^{0}(r,\cdot ) \bigr
\rrVert _{2}^{q} \Bigr)ds \Biggr]
\nonumber
\\
& \leqslant  c \Biggl[ 1 + \int_{0}^{t}(t-s)^{\frac{-3}{4}}
\Bigl(\sup_{0\leqslant r
\leqslant s} \bigl\llVert w(r,\cdot ) \bigr\rrVert
_{2}^{q} \Bigr)ds \Biggr]
\nonumber
\\
& \leqslant  c \Biggl[ 1 + \int_{0}^{t}(t-s)^{\frac{-3}{4}}ds
\Biggr],
\nonumber
\end{align}
which is clearly finite.

Step 2. Let $w_{1}$ and $w_{2}$ be two elements in $\mathbb{H}$. For any
$t\in [0,T]$ we have
%
\begin{align}
\bigl( \bigl\llVert \mathcal{A}w_{1}(t,\cdot ) -
\mathcal{A}w_{2}(t,\cdot ) \bigr\rrVert _{2}^{q}
\bigr) & \leqslant  c\int_{0}^{t}(t-s)^{\frac{-3}{4}}
\bigl( \bigl\llVert \bigl(w _{1}(r,\cdot ) - w_{2}(r,
\cdot )\bigr)u^{0}(r,\cdot ) \bigr\rrVert _{1}^{q}
\bigr)ds
\nonumber
\\
& \leqslant  c\int_{0}^{t}(t-s)^{\frac{-3}{4}}
\bigl( \bigl\llVert w_{1}(s, \cdot ) - w_{2}(s, \cdot )
\bigr\rrVert _{2}^{q} \bigl\llVert u^{0}(r,
\cdot ) \bigr\rrVert _{2}^{q} \bigr)ds
\nonumber
\\
& \leqslant  c\int_{0}^{t}(t-s)^{\frac{-3}{4}}
\bigl( \bigl\llVert w_{1}(s, \cdot ) - w_{2}(s,\cdot )
\bigr\rrVert _{2}^{q} \bigr)ds.
\end{align}
Then, using Fubini's theorem we have
%
\begin{align*}
&\int_{0}^{T}e^{-\lambda t} \bigl( \bigl
\llVert \mathcal{A}w_{1}(t,\cdot ) - \mathcal{A}w_{2}(t,
\cdot ) \bigr\rrVert _{2}^{q} \bigr)dt\\
&\quad  \leqslant  c\int
_{0}^{T}e^{-\lambda t}\int
_{0}^{t}(t-s)^{\frac{-3}{4}} \bigl( \bigl
\llVert w_{1}(s, \cdot ) - w_{2}(s,\cdot ) \bigr\rrVert
_{2}^{q} \bigr)dsdt
\nonumber
\\
&\quad  \leqslant  c\int_{0}^{T}\int
_{s}^{T}e^{-\lambda t}(t-s)^{
\frac{-3}{4}}
\bigl( \bigl\llVert w_{1}(s,\cdot ) - w_{2}(s,\cdot )
\bigr\rrVert _{2}^{q} \bigr)dsdt
\nonumber
\\
&\quad  \leqslant  c\int_{0}^{T} \bigl( \bigl
\llVert w_{1}(s,\cdot ) - w_{2}(s,\cdot ) \bigr\rrVert
_{2}^{q} \bigr)\int_{s}^{T}
e^{-\lambda t}(t-s)^{\frac{-3}{4}}dsdt
\nonumber
\\
&\quad  \leqslant  c \Biggl(\int_{0}^{T}e^{-\lambda r}
r^{\frac{-3}{4}}dr \Biggr) \llVert w_{1}- w_{2}
\rrVert _{\mathbb{H}}^{q}.
\nonumber
\end{align*}
Take $\lambda $ and $T_{0} > 0$ such that
\begin{equation*}
c \int_{0}^{T_{0}}e^{-\lambda r}
r^{\frac{-3}{4}}dr < 1.
\end{equation*}
Then, for $T\leqslant T_{0}$, the operator $\mathcal{A}$ is a
contraction on $\mathbb{H}$. Consequently, for any $v\in S^{N}$, it
admits a unique fixed point $u^{v}\in \mathbb{H}$ which satisfies the
equation (\ref{17}). By concatenation we can construct a solution on
every interval $[0,T]$.

The continuity\index{continuity} of the solution $u^{v}$ follows from the continuity\index{continuity} of the
integrals. For the estimation (\ref{estm}), one can use for
$u^{v}$ the same computations as in (\ref{36}) and Gronwall's lemma.
\end{proof}

In order to prove Lemma \ref{tight2} we have used the following lemma
whose proof can be found in Lemma $3.3$ in
\cite{gyongy1998existence}.

\begin{lem}
\label{19}
For $v\in L^{\infty }([0,T]; L^{1}([0,1]))$, set
\begin{align*}
J(v)(t,x) := \int_{0}^{t}\int_{0}^{1}\partial _{y} G(t,s,x,y)v(s,y)dyds,
\quad t\in [0,T],\  x\in [0,1] .
\end{align*}
Let $\rho \in [1, +\infty [$ and
$q\in [1, \rho [$. Moreover, let $\zeta _{\varepsilon }(t,x)$ be a family
of random fields on $[0,T]\times [0,1]$ such that $\sup_{t\leqslant T}
\|\zeta _{\varepsilon }(t,\cdot )\|_{q} \leqslant \theta _{\varepsilon
}$, where $\theta _{\varepsilon }$ is a finite random variable for every
$\varepsilon $. Assume that the family $\theta _{\varepsilon }$ is
bounded in probability, i.e.,
\begin{equation*}
\lim_{c\longrightarrow +\infty }\sup_{\varepsilon }\mathbb{P}\{ \theta
_{\varepsilon } \geqslant c \} = 0.
\end{equation*}
Then, the family $  (J(\zeta _{\varepsilon })  )_{\varepsilon
> 0}$ is uniformly tight in $C([0,T]; L^{\rho }([0,1]))$.
\end{lem}

We summarize some important proprieties of the sequence $\{K_{n}, n
\geqslant 1\}$ in the following lemma.
%
\begin{lem}
\label{80}
Let $(v_{n})\subset S^{N}$ be a sequence converging weakly in
$\mathcal{H}_{T}$ to an element $v$ in $S^{N}$. The sequence
$\{K_{n}, n\geqslant 1\}$ defined in (\ref{83}) satisfies the
following:
\begin{itemize}%
\item[i)] the sequence $\{K_{n}(t,x), n\geqslant 1\}$ converges to
zero, for any fixed $(t,x)\in [0,T]\times [0,1]$;
\item[ii)] there exists a constant $c(N,T)$ depending on $N$ and
$T$ such that
%
\begin{equation}
\label{81} \sup_{n\geqslant 1}\sup_{t\in [0,T]}
\bigl\llVert K_{n}(t,\cdot ) \bigr\rrVert _{2} \leqslant
c(N, T);
\end{equation}%
\item[iii)]
%
\begin{equation}
\label{82} \lim_{n\longrightarrow \infty }\sup_{t\in [0,T]}
\bigl\llVert K_{n}(t, \cdot ) \bigr\rrVert _{2} = 0.
\end{equation}
\end{itemize}
\end{lem}

\begin{proof} First notice that since
\begin{align*}
\bigl\llVert \mathbf{1}_{[0, t]}(\cdot )G_{t-\cdot }(x,\ast )
\sigma \bigl(u^{0}( \cdot ,\ast )\bigr) \bigr\rrVert
^{2}_{\mathcal{H}_{T}} &:= \int_{0}^{T}
\int_{{0}}^{1}\mathbf{1}_{[0,
t]}(s)G_{t-s}^{2}(x,y)
\sigma ^{2}\bigl(u^{0}(s,y)\bigr)dyds
\\
&\leqslant  c \sup_{x\in [0,1]} \int_{0}^{t}
\int_{{0}}^{1}G_{t-s}
^{2}(x,y)dyds < +\infty ,
\end{align*}
we have $\mathbf{1}_{[0, t]}(\cdot )G_{t-\cdot }(x,\ast )\sigma (u
^{0}(\cdot ,\ast )) \in \mathcal{H}_{T}$ and hence
\begin{eqnarray*}
K_{n}(t,x) = \bigl\langle \mathbf{1}_{[0, t]}(\cdot
)G_{t-\cdot }(x,\ast ) \sigma \bigl(u^{0}(\cdot ,\ast )\bigr)
, v_{n} - v \bigr\rangle _{\mathcal{H}_{T}}.
\end{eqnarray*}

Therefore by the weak convergence\index{weak convergence} of $(v_{n})$ to $v$ in $\mathcal{H}
_{T}$, we get the point \textit{i)} of Lemma \ref{80}.

Now, let us show (\ref{81}) and (\ref{82}). Using the Cauchy--Schwarz
inequality, the boundedness\index{boundedness} of $\sigma $, the facts that $v_{n}$,
$v \in S^{N}$ and Lemma \ref{20}, we have for any $0\leqslant t\leqslant
T$,
%
\begin{align}
\label{65} \bigl\llVert K_{n}(t,\cdot ) \bigr\rrVert
_{2}^{2} & =  \int_{0}^{1}
\Biggl\llvert \int_{0}^{t}\int
_{{0}}^{1}G_{t-s}(x,y)\sigma
\bigl(u^{0}(s,y)\bigr) \bigl(v_{n}(s,y) -v(s,y)
\bigr)dyds \Biggr\rrvert ^{2}dx
\nonumber
\\
& \leqslant  \llVert v_{n} - v \rrVert _{\mathcal{H}_{T}}^{2}
\Biggl(\sup_{x
\in [0,1]} \int_{0}^{t}
\int_{{0}}^{1}\bigl(G_{t-s}(x,y)\sigma
\bigl(u^{0}(s,y)\bigr)\bigr)^{2}dyds \Biggr)
\nonumber
\\
& \leqslant  c(N, T) \Biggl(\sup_{x\in [0,1]} \int
_{0}^{t}\int_{{0}}
^{1} G_{t-s}^{2}(x,y)dyds \Biggr)
\nonumber
\\
& \leqslant  c(N, T),
\end{align}
for some constant $c(N, T)$ depending only on $N$ and $T$, and not on
$n$. This yields (\ref{81}).

It remains to prove (\ref{82}). Following similar arguments as above,
we have, for any $t, t' \in [0,T]$ such that $t \leqslant t'$,
%
\begin{align}
\label{65-ascoli-} \bigl\llVert K_{n}(t,\cdot ) -
K_{n}\bigl(t',\cdot \bigr) \bigr\rrVert
_{2}^{2} & \leqslant  c(N, T)\Biggl(\sup
_{x\in [0,1]} \int_{0}^{t}\int
_{{0}} ^{1} \bigl(G_{t-s}(x,y) -
G_{t'-s}(x,y)\bigr)^{2}dyds
\nonumber
\\
&  \quad{}+ \sup_{x\in [0,1]} \int_{t}^{ t'}
\int_{{0}}^{1} G_{t'-s}^{2}(x,y)
dyds\Biggr)
\nonumber
\\
& \leqslant  c(N, T) \bigl\llvert t-t' \bigr\rrvert
^{1/2}.
\end{align}
According to (\ref{65}) and (\ref{65-ascoli-}), the sequence
$\{K_{n}, n\geqslant 1\}$ is a bounded and H\"{o}lder continuous family
in $C([0,T]; L^{2}([0,1]))$; hence it is a bounded equicontinuous family
and therefore by \textit{i)} of Lemma \ref{80} and the Arzel\`{a}--Ascoli theorem
we get (\ref{82}).
\end{proof}
\end{appendix}

\begin{acknowledgement}

The authors are very thankful to the Editor for her very constructive
criticism from which our final version of the article has benefited.
Many thanks also to the referees for their careful reading and useful
remarks. We are also very indebted to Professors R. Zhang and J. Xiong
for some kind discussions we had about the Burkholder--Davis--Gundy
inequality for SPDEs\index{SPDEs} driven by a space-time white noise.

\end{acknowledgement}


\end{document}